\documentclass[preprint,12pt]{elsarticle}
\usepackage[T1]{fontenc}
\usepackage{lmodern}
\usepackage{amsmath,amssymb,amsthm,mathtools,mathrsfs}
\usepackage{booktabs,tabularx,array,graphicx,microtype}
\usepackage{lineno}
\usepackage[hidelinks]{hyperref}
\biboptions{numbers,sort&compress}
\numberwithin{equation}{section}
\allowdisplaybreaks[2]
\newtheorem{theorem}{Theorem}[section]
\newtheorem{proposition}[theorem]{Proposition}
\newtheorem{lemma}[theorem]{Lemma}
\newtheorem{corollary}[theorem]{Corollary}
\theoremstyle{definition}
\newtheorem{assumption}[theorem]{Assumption}
\theoremstyle{remark}

\newcommand{\R}{\mathbb R}
\newcommand{\E}{\mathbb E}
\newcommand{\K}{\mathcal K}
\newcommand{\Ccal}{\mathcal C}
\newcommand{\Fcal}{\mathcal F}
\newcommand{\Zstar}{\mathcal Z^\star}
\newcommand{\Dcal}{\mathscr D}
\DeclareMathOperator*{\argmin}{arg\,min}
\DeclareMathOperator{\diag}{diag}
\begin{document}
\begin{frontmatter}
\title{Equilibrium bias and convergence in augmented primal--dual dynamics with sampled constraints}
\author[xjtu]{Kang Liu\corref{cor1}}
\ead{kanyo@foxmail.com}
\author[xjtu-auto]{Mengxiao Chen}
\author[polyu]{Siqi Xiong}
\author[xjtu-auto]{Yi Xia}
\author[hust]{Zhan Sun}
\cortext[cor1]{Corresponding author.}

\affiliation[xjtu]{
    organization={School of Future Technology, Xi'an Jiaotong University},
    city={Xi'an},
    country={China}
}

\affiliation[xjtu-auto]{
    organization={School of Automation Science and Engineering, Xi'an Jiaotong University},
    city={Xi'an},
    country={China}
}

\affiliation[polyu]{
    organization={Department of Rehabilitation Sciences, The Hong Kong Polytechnic University},
    city={Hong Kong SAR},
    country={China}
}

\affiliation[hust]{
    organization={School of Integrated Circuits, Huazhong University of Science and Technology},
    city={Wu'Han},
    country={China}
}
\begin{abstract}
This work studies the stability and convergence of augmented primal–dual dynamics when constraint values are estimated from samples. Unbiased constraint observations can produce a biased augmented
multiplier signal, shifting the equilibria of the mean dynamics. For componentwise inequalities, we give a necessary and sufficient condition for preserving the Karush–Kuhn–Tucker (KKT) equilibria and construct a convex example with a locally exponentially stable equilibrium that violates complementarity. To address this bias, constraint values are estimated recursively before forming the augmented multiplier signal. For smooth convex conic problems, a joint energy analysis establishes boundedness of the primal, dual, and estimation states, vanishing estimation error, and almost sure convergence of the primal–dual iterates to a single KKT point under global regularity and bounded conditional second moments. The result allows nonunique solutions and multipliers while keeping the number of samples per iteration fixed. Numerical studies illustrate the predicted equilibrium bias and examine convergence with nonunique KKT points and nonlinear constraints.
\end{abstract}
\end{frontmatter}

\section{Introduction}
\label{sec:intro}

Primal--dual methods solve constrained optimization problems by updating decision variables together with multipliers associated with the constraints \citep{boyd2004}. Their continuous dynamics also provide a basis for studying stability and convergence, as developed by Feijer and Paganini \citep{feijer2010} for primal--dual gradient systems. Augmented Lagrangian methods modify the interaction between the primal and dual variables \citep{rockafellar1976}. For nonlinear convex inequalities, Tang et al.\ \citep{tang2020} established stability results for augmented primal--dual gradient dynamics under strong convexity. These dynamics use constraint values to determine the multiplier feedback. When a constraint is an expectation, however, its value may only be available through noisy observations. It is then necessary to understand how sampling affects the equilibria and convergence of the resulting system.

For inequality constraints, the augmented multiplier signal applies a positive part operation to the current multiplier and a scaled constraint value. An unbiased observation of the constraint can produce a biased signal after this nonlinear operation. With a fixed augmentation scale, decreasing the optimization step does not remove this bias from the mean update direction. Consequently, a KKT point of the original problem may cease to be an equilibrium of the mean dynamics. The sampled dynamics may instead admit a stable equilibrium at which a multiplier is positive while its constraint is strictly satisfied. Stability and feasibility alone therefore do not ensure that the limiting state satisfies complementarity.

A natural way to address this problem is to estimate the constraint value recursively before forming the augmented signal. Wang et al.\ \citep{wang2017} developed auxiliary recursions for tracking inner expectations in stochastic compositional optimization. In the present setting, the quantity being estimated changes with the primal iterate, while the primal update depends on both the estimate and the multiplier. The three states therefore evolve as a coupled system. A convergence analysis must first establish their boundedness, then show that the estimation error vanishes and that the limiting primal--dual state satisfies the original KKT conditions.

For smooth convex conic problems, this work makes two contributions:
\begin{enumerate}
\item We characterize when direct constraint sampling preserves a KKT equilibrium and construct a convex example with a locally exponentially stable non-KKT equilibrium of the mean dynamics.
\item Under global regularity and bounded conditional second moments, we prove joint boundedness, vanishing tracking error, and almost sure point convergence of the recursive iteration, with fixed sample counts and a possibly nonunique KKT set.
\end{enumerate}

The analysis begins with an exact energy identity for the ideal augmented dynamics. Combining this energy with the estimation error establishes stability of the stochastic iteration. An invariant set argument then recovers stationarity from limiting trajectories, and convergence of the distance energies yields convergence to a single KKT point. Numerical studies examine the predicted equilibrium bias, convergence with a nonunique KKT set, and the effect of constraint estimation in nonlinear problems. Comparisons with stochastic primal--dual methods report both the residuals of current iterates and the objective and feasibility errors of their prescribed outputs.

\section{Related work}
\label{sec:related}

Augmented Lagrangian methods provide a classical approach to
constrained convex optimization. Rockafellar \citep{rockafellar1976}
connected these methods with the proximal point algorithm.
For continuous dynamics, Feijer and Paganini \citep{feijer2010}
studied the stability of primal--dual gradient systems, while
Qu and Li \citep{qu2019} established exponential stability under
strong convexity and suitable constraint conditions.
Tang et al.\ \citep{tang2020} extended the analysis to augmented
dynamics with nonlinear convex inequalities, obtaining
semiglobal exponential stability under additional regularity
conditions. These results concern exact constraint feedback;
sampling a constraint before applying the nonlinear augmented
signal changes the vector field being analyzed.

Stochastic primal--dual methods address several forms of
incomplete problem information. For distributed optimization,
Niu et al.\ \citep{niu2019} obtained almost sure convergence
using stochastic averaging gradients under strong convexity.
Du et al.\ \citep{du2025} combined distributed tracking with
geometrically increasing gradient batches to obtain mean square
linear convergence for coupled affine inequalities over
time-varying networks. When the constraints themselves are
expectations, their values must also be estimated.
Zhang et al.\ \citep{zhang2022} incorporated sampled constraints
into a stochastic linearized proximal method of multipliers,
with bounds on objective error and constraint violation.
Yan and Xu \citep{yan2022} developed adaptive primal--dual
stochastic gradient updates with guarantees for averaged
outputs. These methods motivate distinguishing how a constraint
observation enters the update from how the resulting iterates
are evaluated.

Recursive estimation addresses the composition of expectations
with nonlinear maps. Wang et al.\ \citep{wang2017} introduced
stochastic compositional gradient descent, using an auxiliary
recursion to track an inner expectation on a faster time scale.
Subsequent developments improved convergence rates through
acceleration \citep{wangacc2017} and stochastic correction
\citep{chen2021}. Extensions treat compositional constraints
\citep{thomdapu2023,yang2026} and compositional minimax problems
\citep{deng2025}. In the present setting, the tracked constraint
enters an augmented multiplier signal, and the resulting
stochastic direction also depends on an evolving multiplier.

This work focuses on the equilibrium distortion caused by
sampling constraints at a fixed augmentation scale.
It characterizes when a KKT point is preserved and analyzes
recursive estimation as a remedy. The main analytical task is
to establish boundedness of the coupled states and convergence
of the current primal--dual iterates to one KKT point, with
fixed sample counts, no imposed bounded primal or dual domain,
and a possibly nonunique KKT set.

\section{Ideal augmented primal--dual dynamics}
\label{sec:setup}

\subsection{Convex constraints and augmented feedback}

Let $n$ and $m$ be positive integers. Consider the problem
\begin{equation*}
\min_{x\in\R^n} f(x)
~~\text{subject to}~~
c(x)\in-\K,
\end{equation*}
where $x$ is the decision variable,
$f:\R^n\to\R$ is the objective,
$c:\R^n\to\R^m$ is the constraint map, and
$\K\subseteq\R^m$ is a nonempty closed convex cone.
Here $-\K\coloneqq\{-v:v\in\K\}$.
The dual cone of $\K$ is
\[
\Ccal\coloneqq\K^*
\coloneqq
\{u\in\R^m:\langle u,v\rangle\ge0
\text{ for every }v\in\K\},
\]
where $\langle a,b\rangle\coloneqq a^\top b$ denotes the
Euclidean inner product and ${}^\top$ denotes transpose.
The variable $u\in\Ccal$ is the Lagrange multiplier.

Write $c=(c_1,\ldots,c_m)^\top$ and define its Jacobian by
\[
J_c(x)\coloneqq
\left(\frac{\partial c_i(x)}{\partial x_j}\right)_
{\substack{1\le i\le m\\1\le j\le n}}
\in\R^{m\times n}.
\]
The ordinary Lagrangian is
\[
L(x,u)\coloneqq f(x)+\langle u,c(x)\rangle.
\]
The map $c$ is $\K$-convex if
$\langle u,c(\cdot)\rangle$ is convex for every $u\in\Ccal$.
The Karush--Kuhn--Tucker (KKT) set is
\[
\Zstar\coloneqq
\left\{
(x,u)\in\R^n\times\R^m:
\begin{array}{l}
\nabla f(x)+J_c(x)^\top u=0,\\
c(x)\in-\K,~ u\in\Ccal,~
\langle u,c(x)\rangle=0
\end{array}
\right\},
\]
where $\nabla f$ denotes the gradient of $f$.

\begin{assumption}[Convexity and local regularity]
\label{ass:convex}
The function $f$ is convex and the map $c$ is $\K$-convex.
Both are continuously differentiable.
The gradient $\nabla f$ and the Jacobian $J_c$ are locally
Lipschitz, and $\Zstar$ is nonempty.
\end{assumption}

Ordinary componentwise inequalities correspond to
$\K=\R_+^m$, where $\R_+^m$ is the nonnegative orthant.
For $m_e$ equality constraints and $m_i$ inequality constraints,
with $m_e+m_i=m$, take
$\K=\{0\}^{m_e}\times\R_+^{m_i}$.
In this case, $\K$-convexity requires the equality components
of $c$ to be affine.
Expectation constraints have the form
$c(x)=\E[h(x;\xi)]$, where $\xi$ is a random element,
$h(x;\xi)\in\R^m$ is an integrable random constraint vector,
and $\E$ denotes expectation with respect to $\xi$.
The observations introduced later estimate this population
constraint value.

Fix a symmetric positive definite matrix
$D\in\R^{m\times m}$ and a scalar $\kappa>0$.
The matrix $D$ determines the augmentation and projection
metric, while $\kappa$ scales the multiplier dynamics.
For any symmetric positive definite matrix $M$ of the
appropriate dimension, define
\[
\|v\|_M^2\coloneqq v^\top Mv.
\]
Unsubscripted vector norms are Euclidean, and
unsubscripted matrix norms are the corresponding induced
operator norms.

For $w\in\R^m$, define the metric projection onto $\Ccal$ by
\begin{equation}
P_D(w)\coloneqq
\argmin_{v\in\Ccal}\frac12\|v-w\|_{D^{-1}}^2.
\label{eq:projection}
\end{equation}
The minimizer is unique because $\Ccal$ is nonempty, closed,
and convex, and the objective is strictly convex and coercive.

For $(x,u)\in\R^n\times\R^m$, define
\begin{equation}
\begin{aligned}
\lambda(x,u)&\coloneqq P_D(u+Dc(x)),
&
d(x,u)&\coloneqq\lambda(x,u)-u,
\\
G(x,u)&\coloneqq\nabla f(x)+J_c(x)^\top\lambda(x,u),
&
F(x,u)&\coloneqq
\begin{pmatrix}
-G(x,u)\\
\kappa d(x,u)
\end{pmatrix}.
\end{aligned}
\label{eq:signals}
\end{equation}
Here $\lambda(x,u)\in\Ccal$ is the augmented multiplier signal,
$d(x,u)\in\R^m$ is its difference from the current multiplier,
$G(x,u)\in\R^n$ determines the primal descent direction,
and $F(x,u)\in\R^{n+m}$ is the joint vector field.

For $D=\rho I_m$ with $\rho>0$ and
$\K=\R_+^m$, the signal reduces to
$\lambda(x,u)=[u+\rho c(x)]_+$.
Here $I_m$ is the $m\times m$ identity matrix and
$[a]_+$ denotes the vector with components
$[a_i]_+\coloneqq\max\{a_i,0\}$.

The weighted augmented Lagrangian $\mathcal L_D$ satisfies
\begin{equation}
\begin{aligned}
\mathcal L_D(x,u)
&\coloneqq
f(x)+\frac12\left(
\|P_D(u+Dc(x))\|_{D^{-1}}^2
-\|u\|_{D^{-1}}^2
\right),
\\
\nabla_x\mathcal L_D(x,u)&=G(x,u),
~~
\nabla_u\mathcal L_D(x,u)=D^{-1}d(x,u),
\end{aligned}
\label{eq:augmented}
\end{equation}
where $\nabla_x$ and $\nabla_u$ denote gradients with respect
to the indicated variables.
Appendix~\ref{app:projection} derives these identities.

With exact constraint values, the ideal dynamics are
\begin{equation}
\dot x=-G(x,u),
~~
\dot u=\kappa d(x,u),
~~
u(0)\in\Ccal.
\label{eq:ode}
\end{equation}
A dot denotes differentiation with respect to the evolution
time $t\ge0$.
Thus the joint state $z(t)\coloneqq(x(t),u(t))$ satisfies
$\dot z(t)=F(x(t),u(t))$.

For $u\in\Ccal$, the normal cone to $\Ccal$ at $u$ is
\[
N_{\Ccal}(u)\coloneqq
\{s\in\R^m:
\langle s,v-u\rangle\le0
\text{ for every }v\in\Ccal\}.
\]
For $u\notin\Ccal$, set $N_{\Ccal}(u)\coloneqq\varnothing$.

\begin{proposition}[Complementarity and equilibria]
\label{prop:zeros}
The maps in \eqref{eq:signals} satisfy
\begin{equation}
c(x)-D^{-1}d(x,u)
\in N_{\Ccal}(\lambda(x,u)).
\label{eq:normal}
\end{equation}
Moreover, $d(x,u)=0$ if and only if
\(
u\in\Ccal,
c(x)\in-\K,
\langle u,c(x)\rangle=0.
\)
Consequently, $F(x,u)=0$ if and only if
$(x,u)\in\Zstar$.
\end{proposition}

\begin{proof}
For brevity, write $\lambda=\lambda(x,u)$ and
$d=d(x,u)$.
Optimality in \eqref{eq:projection} gives
\[
D^{-1}(u+Dc(x)-\lambda)\in N_{\Ccal}(\lambda),
\]
which is \eqref{eq:normal}.

Since $\Ccal=\K^*$ and $\K$ is a closed convex cone,
\(
N_{\Ccal}(u)=(-\K)\cap u^\perp
~ (u\in\Ccal),
\)
where
$u^\perp\coloneqq\{s\in\R^m:\langle s,u\rangle=0\}$.
If $d=0$, then $\lambda=u\in\Ccal$, and
\eqref{eq:normal} gives feasibility and complementarity.
Conversely, these conditions imply
$c(x)\in N_{\Ccal}(u)$.
The projection optimality condition then gives
$u=P_D(u+Dc(x))$, so $d=0$.
At such a point, $G(x,u)=0$ is exactly the KKT
stationarity condition.
\end{proof}

Define the equilibrium residual
\begin{equation}
\mathcal R(x,u)\coloneqq
\|G(x,u)\|^2+\|d(x,u)\|^2.
\label{eq:residual}
\end{equation}
By Proposition~\ref{prop:zeros},
$\mathcal R(x,u)=0$ exactly at the KKT points.

\subsection{Energy dissipation of the ideal dynamics}

Fix a KKT reference point
$z_*\coloneqq(x_*,u_*)\in\Zstar$ and write
$c_*\coloneqq c(x_*)$.
Define the energy relative to this point by
\begin{equation}
V_*(x,u)\coloneqq
\frac12\|x-x_*\|^2
+\frac1{2\kappa}\|u-u_*\|_{D^{-1}}^2.
\label{eq:energy}
\end{equation}
Throughout this subsection, $\lambda$ and $d$ without
arguments mean $\lambda(x,u)$ and $d(x,u)$.
For a joint state $z=(x,u)$, we also write
$V_*(z)=V_*(x,u)$.

The normal cone relation \eqref{eq:normal} and
$c_*\in N_{\Ccal}(u_*)$ imply
\[
\mathcal N_*(x,u)\coloneqq
\langle\lambda-u_*,
c(x)-c_*-D^{-1}d\rangle\ge0.
\]
The quantity $\mathcal N_*$ is the nonnegative term supplied
by monotonicity of the normal cone.
Since $u-u_*=(\lambda-u_*)-d$, its definition yields
\begin{equation}
\langle u-u_*,D^{-1}d\rangle
=
\langle\lambda-u_*,c(x)-c_*\rangle
-\|d\|_{D^{-1}}^2-\mathcal N_*(x,u).
\label{eq:supply}
\end{equation}

Define the first order remainders of the objective and
constraint map by
\[
\begin{aligned}
B_f(x_*,x)
&\coloneqq
f(x_*)-f(x)-\langle\nabla f(x),x_*-x\rangle,
\\
B_c(x_*,x)
&\coloneqq
c(x_*)-c(x)-J_c(x)(x_*-x),
\end{aligned}
\]
and define the Lagrangian gap at the reference multiplier by
\[
Q_*(x)\coloneqq L(x,u_*)-L(x_*,u_*).
\]
Here $B_f$ and $Q_*$ are scalar valued, while
$B_c(x_*,x)\in\R^m$.
Convexity gives $B_f(x_*,x)\ge0$ and
$B_c(x_*,x)\in\K$.
Also $Q_*(x)\ge0$, because the KKT stationarity condition
makes $x_*$ a global minimizer of the convex function
$L(\cdot,u_*)$.

The total dissipation relative to $z_*$ is
\begin{equation}
\begin{aligned}
\Dcal_*(x,u)\coloneqq{}
B_f(x_*,x)+Q_*(x)
+\langle\lambda,B_c(x_*,x)\rangle
+\|d\|_{D^{-1}}^2+\mathcal N_*(x,u).
\end{aligned}
\label{eq:dissipation}
\end{equation}
As with the energy, write
$\Dcal_*(z)=\Dcal_*(x,u)$ when $z=(x,u)$.

\begin{lemma}[Exact energy identity]
\label{lem:energy}
Under Assumption~\ref{ass:convex}, $\Dcal_*(x,u)\ge0$,
and every solution of \eqref{eq:ode} satisfies
\begin{equation}
\frac{d}{dt}V_*(x(t),u(t))
=-\Dcal_*(x(t),u(t)).
\label{eq:energy_identity}
\end{equation}
\end{lemma}

\begin{proof}
Suppress the time arguments along the trajectory.
Differentiating \eqref{eq:energy} and applying
\eqref{eq:supply} gives
\[
\begin{aligned}
\dot V_*={}
-\langle x-x_*,
\nabla f(x)+J_c(x)^\top\lambda\rangle
+\langle\lambda-u_*,c(x)-c_*\rangle
-\|d\|_{D^{-1}}^2-\mathcal N_*(x,u).
\end{aligned}
\]
Expanding the definitions yields
\[
\begin{aligned}
\langle x-x_*,
\nabla f(x)+J_c(x)^\top\lambda\rangle
-\langle\lambda-u_*,c(x)-c_*\rangle
 =
B_f(x_*,x)+Q_*(x)
+\langle\lambda,B_c(x_*,x)\rangle.
\end{aligned}
\]
Substitution proves \eqref{eq:energy_identity}.
Every term in \eqref{eq:dissipation} is nonnegative:
the first two by convexity, the third by
$\lambda\in\Ccal=\K^*$ and $B_c(x_*,x)\in\K$,
the fourth by positive definiteness of $D$, and the last
by normal cone monotonicity.
\end{proof}

\begin{theorem}[Convergence of the ideal dynamics]
\label{thm:ideal}
Under Assumption~\ref{ass:convex}, every initial condition
$x(0)\in\R^n$ and $u(0)\in\Ccal$ generates a unique solution
of \eqref{eq:ode} for all $t\ge0$.
The solution is bounded and converges to a point of $\Zstar$.
\end{theorem}

\begin{proof}
The vector field $F$ is locally Lipschitz, so a unique local
solution exists.
Along this solution, the multiplier satisfies
\[
u(t)=e^{-\kappa t}u(0)
+\kappa\int_0^t
e^{-\kappa(t-s)}\lambda(x(s),u(s))\,ds
\in\Ccal,
\]
because $\Ccal$ is a closed convex cone.
By Lemma~\ref{lem:energy}, the joint state remains in the
compact sublevel set
\[
\left\{
(x,u)\in\R^n\times\Ccal:
V_*(x,u)\le V_*(x(0),u(0))
\right\}.
\]
The solution is therefore bounded and extends to all
$t\ge0$.

To identify the invariant part of the zero dissipation set,
consider a trajectory along which $\Dcal_*=0$.
Since all terms in \eqref{eq:dissipation} are nonnegative,
the term $\|d\|_{D^{-1}}^2$ gives $d=0$.
Thus $u$ is constant and $\lambda=u$.
Proposition~\ref{prop:zeros} gives feasibility and
complementarity.
Using these conditions at $(x,u)$ and at $(x_*,u_*)$,
the definition of $\mathcal N_*$ becomes
\[
\mathcal N_*(x,u)
=-\langle u,c_*\rangle-\langle u_*,c(x)\rangle.
\]
Both terms on the right are nonnegative.
Their sum is zero, hence
$\langle u_*,c(x)\rangle=0$.
Since $Q_*(x)=0$ and
$\langle u_*,c_*\rangle=0$, it follows that
$f(x)=f(x_*)$.

Along this trajectory, complementarity and constancy of $u$
give
\[
L(x(t),u)=f(x_*),
~~
\frac{d}{dt}L(x(t),u)
=-\|G(x(t),u)\|^2.
\]
Therefore $G=0$.
The largest invariant subset of the zero dissipation set
consists exactly of KKT equilibria.
LaSalle's invariance principle \citep{lasalle1976} then
places every accumulation point of the trajectory in $\Zstar$.

Let
$z^\infty\coloneqq(x^\infty,u^\infty)\in\Zstar$
be one accumulation point.
The energy relative to this point is
\[
V_\infty(x,u)\coloneqq
\frac12\|x-x^\infty\|^2
+\frac1{2\kappa}\|u-u^\infty\|_{D^{-1}}^2.
\]
Lemma~\ref{lem:energy}, applied with reference point
$z^\infty$, shows that $V_\infty(x(t),u(t))$ is
nonincreasing.
It tends to zero along a subsequence of times approaching
infinity, and therefore tends to zero along the full
trajectory.
Consequently, $(x(t),u(t))\to(x^\infty,u^\infty)$.
\end{proof}

The invariant set argument is needed because zero
dissipation at one state does not imply stationarity.
For example, take $n=m=1$, $\K=\R_+$,
$f(x)=x^2/2$, $c(x)=x$, $D=1$, and the KKT reference
point $(x_*,u_*)=(0,0)$.
At a state with $x=0$ and $u>0$, one has
$\Dcal_*(x,u)=0$ but $G(x,u)=u\ne0$.

If $f$ is $\mu$-strongly convex for some $\mu>0$, meaning
that
\[
f(y)\ge
f(x)+\langle\nabla f(x),y-x\rangle
+\frac{\mu}{2}\|y-x\|^2
~~\text{for all }x,y\in\R^n,
\]
then \eqref{eq:dissipation} also gives
\begin{equation}
\Dcal_*(x,u)\ge
\mu\|x-x_*\|^2+\|d(x,u)\|_{D^{-1}}^2.
\label{eq:strong}
\end{equation}
Indeed, $B_f(x_*,x)\ge\mu\|x-x_*\|^2/2$.
The function $L(\cdot,u_*)$ is also $\mu$-strongly convex
and is minimized at $x_*$, so
$Q_*(x)\ge\mu\|x-x_*\|^2/2$.

\section{Equilibrium bias from sampled constraints}
\label{sec:distortion}

\subsection{Effect of direct constraint observations}
\label{sec:bias}

Consider componentwise inequalities, with $\K=\R_+^m$
and $D=\rho I_m$ for a fixed augmentation scale $\rho>0$.
At a fixed state $(x,u)\in\R^n\times\R_+^m$, suppose the
constraint observation is
\(
\widehat c\coloneqq c(x)+\epsilon,
\)
where $\epsilon\in\R^m$ is an integrable random error with
$\E\epsilon=0$.
Throughout this subsection, expectations are conditional
on the fixed state; the error distribution may depend on
that state.

Define the exact signal argument and its scaled error by
\[
s\coloneqq u+\rho c(x),
~~
\varepsilon\coloneqq\rho\epsilon.
\]
Direct sampling replaces the exact multiplier signal
$\lambda(x,u)=[s]_+$ by
\[
\widehat\lambda\coloneqq[s+\varepsilon]_+.
\]
With exact objective gradients and constraint Jacobians,
the resulting mean vector field is
\[
\bar F(x,u)\coloneqq
\begin{pmatrix}
-\nabla f(x)-J_c(x)^\top\E\widehat\lambda\\
\kappa(\E\widehat\lambda-u)
\end{pmatrix}.
\]
Thus the mean dynamics satisfy
$(\dot x,\dot u)=\bar F(x,u)$.
The following scalar calculation characterizes when a KKT point
remains an equilibrium of the mean dynamics.

\begin{proposition}[Bias and preservation of equilibria]
\label{prop:bias}
Let $\varepsilon$ be an integrable scalar random variable
with $\E\varepsilon=0$.
For a scalar argument $s\in\R$, define
\begin{equation}
b(s)\coloneqq
\E[s+\varepsilon]_+-[s]_+
=
\frac{\E|s+\varepsilon|-|s|}{2}.
\label{eq:bias}
\end{equation}
Then $b(s)\ge0$, and
\begin{equation}
b(s)=0
~\Longleftrightarrow~
\begin{cases}
s+\varepsilon\ge0\text{ almost surely},&s>0,\\
s+\varepsilon\le0\text{ almost surely},&s<0,\\
\varepsilon=0\text{ almost surely},&s=0.
\end{cases}
\label{eq:bias_zero}
\end{equation}

For the vector constraint problem, let
$(x_*,u_*)\in\Zstar$ and
$s_*\coloneqq u_*+\rho c(x_*)$.
With exact derivatives and unbiased constraint observations,
$(x_*,u_*)$ is an equilibrium of $\bar F$ if and only if
\eqref{eq:bias_zero} holds for every pair
$(s_{*,i},\varepsilon_i)$, $i=1,\ldots,m$, under the
error distribution at $(x_*,u_*)$.
\end{proposition}

\begin{proof}
For any scalar $t$, the identity
$[t]_+=(t+|t|)/2$, together with $\E\varepsilon=0$,
gives \eqref{eq:bias}.
Convexity of the absolute value gives $b(s)\ge0$.

For $s>0$, the same identity yields
\(
b(s)=\E[-s-\varepsilon]_+.
\)
For $s<0$, one has
\(
b(s)=\E[s+\varepsilon]_+.
\)
At $s=0$,
\(
b(0)=\frac12\E|\varepsilon|.
\)
In each case, the nonnegative random variable inside the
expectation has zero expectation exactly when it is zero
almost surely.
This proves \eqref{eq:bias_zero}.

For the vector problem, define the bias at the KKT point by
\[
b_{*,i}\coloneqq
\E[s_{*,i}+\varepsilon_i]_+-[s_{*,i}]_+,
~~
\boldsymbol b_*\coloneqq(b_{*,1},\ldots,b_{*,m})^\top.
\]
The KKT conditions give
$[s_*]_+=u_*$ and
$\nabla f(x_*)+J_c(x_*)^\top u_*=0$.
Consequently,
\[
\bar F(x_*,u_*)=
\begin{pmatrix}
-J_c(x_*)^\top\boldsymbol b_*\\
\kappa\boldsymbol b_*
\end{pmatrix}.
\]
Since $\kappa>0$, the multiplier component vanishes
exactly when $\boldsymbol b_*=0$.
The primal component then also vanishes.
Applying the scalar criterion to each component completes
the proof.
\end{proof}

The condition requires the perturbed argument to remain
on the same linear branch of the positive part map.
For an active constraint with $u_{*,i}>0$,
one has $c_i(x_*)=0$ and $s_{*,i}=u_{*,i}>0$.
For an inactive constraint, $c_i(x_*)<0$ and
$u_{*,i}=0$, so $s_{*,i}=\rho c_i(x_*)<0$.
At a weakly active constraint, meaning
$c_i(x_*)=u_{*,i}=0$, the argument is zero.
Any centered error that is nonzero with positive probability
then produces a strictly positive bias.
The componentwise criterion does not require independence
between the constraint errors.

\begin{corollary}[Two noise distributions]
\label{cor:distributions}
Let $\tau>0$ denote the amplitude or standard deviation
of the scalar signal error.
If $\varepsilon$ takes the values $\pm\tau$ with equal
probability, then
\[
b(s)=\frac{[\tau-|s|]_+}{2}.
\]
If $\varepsilon\sim N(0,\tau^2)$, where $N(0,\tau^2)$
denotes the normal distribution with mean zero and variance
$\tau^2$, then
\begin{equation}
b(s)=
\tau\phi(|s|/\tau)-|s|\Phi(-|s|/\tau)>0
~~\text{for every }s\in\R,
\label{eq:gaussian}
\end{equation}
where
\[
\phi(z)\coloneqq\frac{1}{\sqrt{2\pi}}e^{-z^2/2},
~~
\Phi(z)\coloneqq\int_{-\infty}^{z}\phi(v)\,dv
\]
are the standard normal density and distribution function.

For a scalar constraint value estimated by the mean of
$B$ independent unbiased Gaussian observations, where
$B$ is a positive integer and each observation has
standard deviation $\sigma>0$, the signal error has
standard deviation $\tau=\rho\sigma/\sqrt B$.
In particular,
\[
b(0)=\frac{\rho\sigma}{\sqrt{2\pi B}}.
\]
\end{corollary}

\begin{proof}
For the two point distribution, expand
\[
\frac{[s+\tau]_++[s-\tau]_+}{2}-[s]_+.
\]
For a standard normal random variable $Z\sim N(0,1)$,
integration gives
\[
\E[s+\tau Z]_+
=
\int_{-s/\tau}^{\infty}(s+\tau z)\phi(z)\,dz
=
s\Phi(s/\tau)+\tau\phi(s/\tau).
\]
Subtracting $[s]_+$ and using symmetry of the normal
distribution yields \eqref{eq:gaussian}.
Strict positivity follows from
Proposition~\ref{prop:bias}, because $s+\tau Z$ has
positive probability on both sides of zero.
The batch formula follows from the variance
$\sigma^2/B$ of the sample mean.
\end{proof}

The next example shows that the bias can lead to a stable
equilibrium that violates the KKT conditions.

\begin{proposition}[A stable incorrect equilibrium]
\label{prop:counterexample}
Consider
\[
\min_{x\in\R}\frac12x^2-hx
~~\text{subject to}~~ x\le0,
\]
where $h$ and $\tau$ are constants satisfying
$0<h<\tau$.
Take $D=1$, any $\kappa>0$, and a constraint observation
error $\epsilon$ that takes the values $\pm\tau$ with
equal probability.
The unique KKT point is $(0,h)$.
The mean dynamics with direct sampling also have the
locally exponentially stable equilibrium
\[
\bar x=\frac{h-\tau}{2},
~~
\bar u=\frac{h+\tau}{2},
\]
which violates complementarity.
\end{proposition}

\begin{proof}
Define the expected multiplier signal by
\[
\ell(x,u)\coloneqq\E[u+x+\epsilon]_+.
\]
On the open region $-\tau<u+x<\tau$, it satisfies
\[
\ell(x,u)=\frac{u+x+\tau}{2}.
\]
The mean dynamics in this region are therefore
\[
\dot x=h-x-\ell(x,u),
~~
\dot u=\kappa(\ell(x,u)-u).
\]
At an equilibrium, the second equation gives
$\ell(x,u)=u$, and the first gives $x+u=h$.
Substitution yields the stated pair $(\bar x,\bar u)$.

Since $\bar x+\bar u=h\in(0,\tau)$, this pair lies
strictly inside the region where the displayed expression
for $\ell$ holds.
Moreover, $\bar x<0$ and $\bar u>0$, so
$\bar x\bar u\ne0$ and complementarity fails.
The Jacobian of the mean vector field there is
\[
\begin{pmatrix}
-3/2&-1/2\\
\kappa/2&-\kappa/2
\end{pmatrix}.
\]
Its trace is $-(3+\kappa)/2<0$ and its determinant is
$\kappa>0$.
Both eigenvalues therefore have strictly negative real
parts, proving local exponential stability.
At the original KKT point,
\[
\ell(0,h)-h=\frac{\tau-h}{2}>0,
\]
so $(0,h)$ is not an equilibrium of the mean dynamics.
\end{proof}

The augmentation scale $\rho$ is fixed in the signal
$[u+\rho\widehat c]_+$.
In the ordinary projected multiplier update
\[
u_{k+1}=[u_k+\alpha_k\widehat c_k]_+,
\]
the constraint observation $\widehat c_k$ is instead
scaled by the optimization step $\alpha_k$.
Here $k=0,1,\ldots$ is the iteration index and
$\alpha_k>0$ decreases to zero.
The numerical studies include this projected update
alongside direct sampling of the augmented signal.

\section{Recursive constraint estimation}
\label{sec:method}

To address the equilibrium bias, we maintain an estimate
of the constraint value and use it to form the augmented
multiplier signal.
At iteration $k$, let $x_k\in\R^n$ be the primal state,
$u_k\in\Ccal$ the multiplier, and $y_k\in\R^m$ an estimate
of $c(x_k)$.
The optimization update uses $y_k$.
A fresh constraint observation at the resulting point
$x_{k+1}$ then updates the estimate.

The quantity being estimated changes as the primal state
moves, while that movement depends on the estimate and
the multiplier.
The analysis therefore treats the three states together.

\subsection{Regularity and observations}

The explicit stochastic iteration uses the following
global regularity assumptions.

\begin{assumption}[Global regularity]
\label{ass:growth}
There are finite nonnegative constants $L_f$, $M_c$,
and $L_J$ such that, for every $x,z\in\R^n$,
\[
\begin{aligned}
\|\nabla f(x)-\nabla f(z)\|
&\le L_f\|x-z\|,\\
\|J_c(x)\|&\le M_c,\\
\|J_c(x)-J_c(z)\|
&\le L_J\|x-z\|.
\end{aligned}
\]
Here $L_f$ is a Lipschitz constant for the objective
gradient, $M_c$ bounds the constraint Jacobian, and
$L_J$ is a Lipschitz constant for that Jacobian.
\end{assumption}

In particular, the bound on $J_c$ makes $c$ globally
Lipschitz.
The constraint assumptions hold for affine maps and,
for componentwise inequalities, smooth convex functions
with bounded and Lipschitz gradients, such as softplus
functions of affine arguments.
These conditions allow the update directions to be
controlled before boundedness of the iterates is known.

Let $\Fcal_k$ be the information available before the
derivative observations at iteration $k$.
It contains the initial states and all previous
observations, and $(x_k,u_k,y_k)$ is
$\Fcal_k$-measurable.
Write
\[
\E_k[\cdot]\coloneqq\E[\cdot\mid\Fcal_k].
\]
The derivative observations are
$\hat g_k\in\R^n$ and $\hat J_k\in\R^{m\times n}$,
which estimate $\nabla f(x_k)$ and $J_c(x_k)$,
respectively.

\begin{assumption}[Conditional observations]
\label{ass:oracle}
Let $B_g$, $B_J$, and $B_c$ be fixed positive integers
denoting the sample counts for the objective gradient,
constraint Jacobian, and constraint value.
There are finite nonnegative constants
$\sigma_g$, $\sigma_J$, and $\sigma_c$ such that,
for every $k$,
\[
\begin{aligned}
\E_k\hat g_k&=\nabla f(x_k),
&
\E_k\|\hat g_k-\nabla f(x_k)\|^2
&\le\frac{\sigma_g^2}{B_g},
\\
\E_k\hat J_k&=J_c(x_k),
&
\E_k\|\hat J_k-J_c(x_k)\|^2
&\le\frac{\sigma_J^2}{B_J}.
\end{aligned}
\]
These constants bound the conditional second moments
of the observation errors.

Let $\Fcal_k^+$ be the information obtained by adding
$\hat g_k$, $\hat J_k$, and the resulting states
$x_{k+1}$ and $u_{k+1}$ to $\Fcal_k$.
The constraint observation at $x_{k+1}$ is
\[
\hat c_{k+1}\coloneqq
c(x_{k+1})+\varepsilon_{k+1}\in\R^m,
\]
where $\varepsilon_{k+1}$ is the constraint observation
error and satisfies
\[
\E[\varepsilon_{k+1}\mid\Fcal_k^+]=0,
~~
\E[\|\varepsilon_{k+1}\|^2\mid\Fcal_k^+]
\le\frac{\sigma_c^2}{B_c}.
\]
The initial states satisfy
\[
\E\bigl[\|x_0\|^2+\|u_0\|^2+\|y_0\|^2\bigr]<\infty,
~~
u_0\in\Ccal~\text{almost surely}.
\]
\end{assumption}

The derivative observations may be correlated with each
other.
The constraint observation must be conditionally unbiased
after the new primal point has been determined.
Evaluating it at $x_{k+1}$ using a fresh independent sample
batch is one way to satisfy this requirement.
The observation $\hat c_{k+1}$ is then included in
$\Fcal_{k+1}$.

\subsection{Iteration and convergence theorem}

Let $\alpha_k>0$ be the optimization step and
$\gamma_k>0$ the estimation gain.
Starting from $(x_0,u_0,y_0)$, define
\begin{equation}
\begin{aligned}
\lambda_k&\coloneqq P_D(u_k+Dy_k),
~~
d_k\coloneqq\lambda_k-u_k,
\\
x_{k+1}&\coloneqq
x_k-\alpha_k(\hat g_k+\hat J_k^\top\lambda_k),
\\
u_{k+1}&\coloneqq u_k+\kappa\alpha_k d_k,
\\
y_{k+1}&\coloneqq
(1-\gamma_k)y_k+\gamma_k\hat c_{k+1}.
\end{aligned}
\label{eq:algorithm}
\end{equation}
Here $\lambda_k$ is the multiplier signal formed from
the estimate $y_k$, and $d_k$ is its difference from
the current multiplier.
The first three lines determine the new optimization
state; the last line uses the subsequent constraint
observation at $x_{k+1}$.

Take $0<\kappa\alpha_k\le1$ and $0<\gamma_k\le1$.
Since
\[
u_{k+1}=(1-\kappa\alpha_k)u_k
+\kappa\alpha_k\lambda_k,
\]
the multiplier remains in the convex cone $\Ccal$.
The initial estimate $y_0$ need not be unbiased.

Define the tracking error by
\[
e_k\coloneqq y_k-c(x_k)\in\R^m.
\]
Its convergence is coupled to that of the primal and
dual states.

\begin{theorem}[Almost sure point convergence]
\label{thm:main}
Suppose Assumptions~\ref{ass:convex},
\ref{ass:growth}, and~\ref{ass:oracle} hold.
Choose a scalar $\vartheta\in(0,1/4)$ and a fixed
time scale parameter $\tau_0\ge1$.
Define the decay exponents
\[
a\coloneqq\frac34+\vartheta,
~~
b\coloneqq\frac12+\vartheta,
\]
and choose the step sequences
\[
\alpha_k\coloneqq
\alpha_0(1+k/\tau_0)^{-a},
~~
\gamma_k\coloneqq
\gamma_0(1+k/\tau_0)^{-b},
\]
where the initial step and gain satisfy
$0<\alpha_0\le1/\kappa$ and $0<\gamma_0\le1$.

Then the iteration \eqref{eq:algorithm} satisfies
\begin{equation}
\sup_{k\ge0}\E\bigl[
\|x_k\|^2+\|u_k\|^2+\|e_k\|^2
\bigr]<\infty,
\label{eq:moments}
\end{equation}
and the sequence $(x_k,u_k,y_k)$ is almost surely bounded.
There is a random KKT point
$(x^\infty,u^\infty)\in\Zstar$ such that
\[
(x_k,u_k)\longrightarrow(x^\infty,u^\infty),
~~
e_k\longrightarrow0,
~~
\lambda_k\longrightarrow u^\infty
~\text{almost surely}.
\]
Moreover,
\[
\E\|e_k\|^2\longrightarrow0,
~~
\mathcal R(x_k,u_k)\longrightarrow0
~\text{almost surely},
\]
where $\mathcal R$ is the equilibrium residual defined
in \eqref{eq:residual}.
\end{theorem}

The theorem establishes convergence to a KKT point with
fixed sample counts, including problems with nonunique
solutions and multipliers.
The parameter $\tau_0$ controls the initial decay of
the steps without changing their asymptotic exponents.
In particular,
\[
\frac{\alpha_k}{\gamma_k}
=
\frac{\alpha_0}{\gamma_0}
(1+k/\tau_0)^{-1/4}\longrightarrow0.
\]
The estimation gain therefore decreases more slowly than
the optimization step.
Section~\ref{sec:proof} uses this separation to establish
joint boundedness, prove that the tracking error vanishes,
and identify the limit of the primal--dual iterates.
\section{Convergence analysis}
\label{sec:proof}

The proof combines the energy identity of the ideal dynamics
with a recursion for the constraint tracking error.
We first derive estimates that hold without assuming bounded
iterates. A weighted sum of the two errors then establishes
joint stability. Finally, limiting trajectories recover the
KKT conditions and yield convergence to a single point.

Throughout this section, fix a KKT reference point
$z_*=(x_*,u_*)\in\Zstar$ and write $c_*=c(x_*)$.
Set $z_k\coloneqq(x_k,u_k)$ and
$V_k\coloneqq V_*(x_k,u_k)$.
For a joint state $z=(x,u)$, use the abbreviations
$F(z)=F(x,u)$ and $\Dcal_*(z)=\Dcal_*(x,u)$.
The tracking error remains
$e_k=y_k-c(x_k)$.
Constants denoted by $C$ are finite and positive.
They may change between inequalities and depend on the
problem data, the fixed reference point, and the fixed
algorithm parameters, but not on the iteration index $k$.

\subsection{Growth and coupled error estimates}

\begin{lemma}[Direction growth]
\label{lem:growth}
Under Assumptions~\ref{ass:convex} and~\ref{ass:growth},
there is a constant $C$ such that, for every
$(x,u)\in\R^n\times\Ccal$,
\begin{equation}
\|G(x,u)\|^2+\|d(x,u)\|^2
\le C V_*(x,u).
\label{eq:growth}
\end{equation}
For any constraint perturbation $e\in\R^m$, define
$\lambda_e\coloneqq P_D(u+D(c(x)+e))$.
Then
\begin{equation}
\begin{aligned}
\|\nabla f(x)+J_c(x)^\top\lambda_e\|^2
+\|\lambda_e-u\|^2\le C\bigl(V_*(x,u)+\|e\|^2\bigr).
\end{aligned}
\label{eq:perturbed_growth}
\end{equation}
\end{lemma}

\begin{proof}
The bound on $J_c$ makes $c$ globally Lipschitz.
Metric nonexpansiveness of $P_D$ and
$P_D(u_*+Dc_*)=u_*$ therefore give
\[
\|\lambda(x,u)-u_*\|
\le C\bigl(\|u-u_*\|+\|x-x_*\|\bigr).
\]
It follows that
\[
\|d(x,u)\|
\le\|\lambda(x,u)-u_*\|+\|u-u_*\|
\le C\sqrt{V_*(x,u)}.
\]

Using the KKT stationarity condition, write
\[
\begin{aligned}
G(x,u)={}&
\nabla f(x)-\nabla f(x_*)
+J_c(x)^\top(\lambda(x,u)-u_*)\\
&+[J_c(x)-J_c(x_*)]^\top u_*.
\end{aligned}
\]
Assumption~\ref{ass:growth} bounds each term by a
constant times $\sqrt{V_*(x,u)}$, proving
\eqref{eq:growth}.

Nonexpansiveness also gives
\[
\|\lambda_e-\lambda(x,u)\|\le C\|e\|.
\]
Combining this inequality with the bound on $J_c$
and \eqref{eq:growth} proves
\eqref{eq:perturbed_growth}.
\end{proof}

Separate the error in the multiplier signal from the
noise in the derivative observations by defining
\[
\begin{aligned}
p_k&\coloneqq\lambda_k-\lambda(x_k,u_k)\in\R^m,\\
\xi_k&\coloneqq
\hat g_k-\nabla f(x_k)
+(\hat J_k-J_c(x_k))^\top\lambda_k
\in\R^n.
\end{aligned}
\]
Thus $p_k$ is the signal perturbation caused by the
constraint estimate, and $\xi_k$ is the derivative noise
in the primal update.
Since $\lambda_k$ is $\Fcal_k$-measurable,
Assumption~\ref{ass:oracle} implies
\begin{equation}
\begin{aligned}
\E_k\xi_k=0,~~
\|p_k\|\le C\|e_k\|,~~
\E_k\|\xi_k\|^2
\le C(1+V_k+\|e_k\|^2).
\end{aligned}
\label{eq:noise}
\end{equation}
For the last estimate, use
$\|\lambda_k\|^2\le C(1+V_k+\|e_k\|^2)$
and bound the squared norm of the sum defining $\xi_k$
by twice the sum of the squared norms.
This argument permits correlation between
$\hat g_k$ and $\hat J_k$.

\begin{lemma}[Coupled drift estimates]
\label{lem:drifts}
Under the assumptions of Theorem~\ref{thm:main},
let $\delta_k\in(0,1]$ be any deterministic auxiliary
parameter.
For every $k$ such that $\alpha_k\le1$,
\begin{equation}
\begin{aligned}
\E_k V_{k+1}\le{}
(1+C\alpha_k^2+C\alpha_k\delta_k)V_k
-\alpha_k\Dcal_*(z_k)
+C\frac{\alpha_k}{\delta_k}\|e_k\|^2
+C\alpha_k^2,
\end{aligned}
\label{eq:v_drift}
\end{equation}
and
\begin{equation}
\begin{aligned}
\E_k\|e_{k+1}\|^2\le{}&
\left(1-\gamma_k
+C\frac{\alpha_k^2}{\gamma_k}\right)\|e_k\|^2\\
&+C\frac{\alpha_k^2}{\gamma_k}(1+V_k)
+C\gamma_k^2.
\end{aligned}
\label{eq:e_drift}
\end{equation}
The constants can be chosen independently of
$\delta_k$ and $k$.
\end{lemma}

\begin{proof}
Relative to the ideal vector field $F(z_k)$, the
conditional mean perturbation is
$(-J_c(x_k)^\top p_k,\kappa p_k)$,
and the centered perturbation is $(-\xi_k,0)$.
The quadratic expansion of $V_*$ has ideal linear term
$-\alpha_k\Dcal_*(z_k)$ by
Lemma~\ref{lem:energy}.
The remaining mean linear term is
\[
\alpha_k\left[
-\langle x_k-x_*,J_c(x_k)^\top p_k\rangle
+\langle u_k-u_*,D^{-1}p_k\rangle
\right].
\]
Its absolute value is at most
$C\alpha_k\sqrt{V_k}\|e_k\|$.
The centered linear term has zero conditional expectation.
Lemma~\ref{lem:growth} and \eqref{eq:noise} bound the
quadratic terms, giving
\[
\begin{aligned}
\E_k V_{k+1}\le{}&
V_k-\alpha_k\Dcal_*(z_k)
+C\alpha_k\sqrt{V_k}\|e_k\|\\
&+C\alpha_k^2(1+V_k+\|e_k\|^2).
\end{aligned}
\]
Apply Young's inequality in the form
\[
2\sqrt{V_k}\|e_k\|
\le\delta_k V_k+\delta_k^{-1}\|e_k\|^2.
\]
Since $\alpha_k\le1$ and $\delta_k\le1$,
$\alpha_k^2\le\alpha_k/\delta_k$.
This proves \eqref{eq:v_drift}.

For the tracking error, \eqref{eq:algorithm} gives the
exact recursion
\[
e_{k+1}
=(1-\gamma_k)
\bigl[e_k+c(x_k)-c(x_{k+1})\bigr]
+\gamma_k\varepsilon_{k+1}.
\]
Define
$h_k\coloneqq c(x_k)-c(x_{k+1})\in\R^m$,
the change in the population constraint value during
the primal update.
Global Lipschitz continuity of $c$, together with
Lemma~\ref{lem:growth} and \eqref{eq:noise}, gives
\[
\E_k\|h_k\|^2
\le C\alpha_k^2(1+V_k+\|e_k\|^2).
\]
Both $e_k$ and $h_k$ are $\Fcal_k^+$-measurable.
The conditional mean assumption on
$\varepsilon_{k+1}$ therefore eliminates its cross term,
so
\[
\E_k\|e_{k+1}\|^2
\le
(1-\gamma_k)^2\E_k\|e_k+h_k\|^2
+C\gamma_k^2.
\]
Using
\[
\|e_k+h_k\|^2
\le
(1+\gamma_k/2)\|e_k\|^2
+(1+2/\gamma_k)\|h_k\|^2
\]
and
\[
(1-\gamma)^2(1+\gamma/2)\le1-\gamma
~~(0<\gamma\le1)
\]
proves \eqref{eq:e_drift}.
\end{proof}

The two estimates describe the dependence between
optimization and estimation.
The tracking error contributes to the energy drift,
while movement of the primal variable contributes to
the tracking error.
Neither estimate requires bounded iterates.

\subsection{A joint energy for stability}

The positive term
$C\alpha_k\|e_k\|^2/\delta_k$
in \eqref{eq:v_drift} can be controlled by the
contraction term in \eqref{eq:e_drift}.
To match their coefficients, define
\[
\begin{aligned}
r&\coloneqq\frac14-\frac{\vartheta}{2},
&
\delta_k&\coloneqq(1+k/\tau_0)^{-r},\\
w_k&\coloneqq
L_0\frac{\alpha_k}{\delta_k\gamma_k},
&
W_k&\coloneqq V_k+w_k\|e_k\|^2,
\end{aligned}
\]
where $\vartheta$ and $\tau_0$ are the parameters in
Theorem~\ref{thm:main}.
The exponent $r$ determines the auxiliary sequence
$\delta_k$, and $L_0>0$ is a proof constant chosen below.
The scalar $W_k$ is the joint energy.
Its weight satisfies
\[
w_k=
L_0\frac{\alpha_0}{\gamma_0}
(1+k/\tau_0)^{-\vartheta/2}.
\]
In particular, $w_k$ is positive and nonincreasing,
and $w_{k+1}/w_k\to1$.

\begin{lemma}[Joint energy estimate]
\label{lem:joint}
There exist constants $L_0,c_0>0$ and a deterministic
integer $k_0\ge0$ such that, for all $k\ge k_0$,
\begin{equation}
\begin{aligned}
\E_k W_{k+1}\le{}&
(1+\eta_k)W_k-\alpha_k\Dcal_*(z_k)\\
&-c_0\frac{\alpha_k}{\delta_k}\|e_k\|^2+q_k,
\end{aligned}
\label{eq:joint_drift}
\end{equation}
where the deterministic nonnegative sequences
$\eta_k$ and $q_k$ are
\begin{equation}
\begin{aligned}
\eta_k&\coloneqq
C\left(
\alpha_k^2+\alpha_k\delta_k
+\frac{\alpha_k^3}{\delta_k\gamma_k^2}
\right),\\
q_k&\coloneqq
C\left(
\alpha_k^2
+\frac{\alpha_k^3}{\delta_k\gamma_k^2}
+\frac{\alpha_k\gamma_k}{\delta_k}
\right).
\end{aligned}
\label{eq:coefficients}
\end{equation}
They satisfy
$\sum_{k=0}^{\infty}\eta_k<\infty$ and
$\sum_{k=0}^{\infty}q_k<\infty$.
\end{lemma}

\begin{proof}
Multiply \eqref{eq:e_drift} by $w_{k+1}$ and add
\eqref{eq:v_drift}.
Relative to $W_k$, the resulting coefficient of
$\|e_k\|^2$ is at most
\[
(w_{k+1}-w_k)-w_{k+1}\gamma_k
+C\frac{\alpha_k}{\delta_k}
+Cw_{k+1}\frac{\alpha_k^2}{\gamma_k}.
\]
The first term is nonpositive.
For all sufficiently large $k$,
$w_{k+1}\ge w_k/2$.
Since $w_{k+1}\le w_k$, the remaining terms are bounded
above by
\[
\frac{\alpha_k}{\delta_k}
\left[
-\frac{L_0}{2}
+C+CL_0\left(\frac{\alpha_k}{\gamma_k}\right)^2
\right].
\]
Choose $L_0$ sufficiently large.
Because $\alpha_k/\gamma_k\to0$, the expression in
brackets is then bounded above by a negative constant
$-c_0$ for all sufficiently large $k$.

The additional coefficient of $V_k$ is bounded by
\[
C\left(
\alpha_k^2+\alpha_k\delta_k
+\frac{\alpha_k^3}{\delta_k\gamma_k^2}
\right).
\]
The additive terms are bounded by
\[
C\left(
\alpha_k^2
+\frac{\alpha_k^3}{\delta_k\gamma_k^2}
+\frac{\alpha_k\gamma_k}{\delta_k}
\right).
\]
Using $V_k\le W_k$ gives \eqref{eq:joint_drift},
after increasing $k_0$ if necessary to ensure
$\alpha_k\le1$.

The decay exponents associated with these terms are
\[
\begin{aligned}
a+r&=1+\vartheta/2,
&
3a-r-2b&=1+3\vartheta/2,\\
a+b-r&=1+5\vartheta/2,
&
2a-b&=1+\vartheta,\\
2b&=1+2\vartheta,
&
2a&=3/2+2\vartheta,
\end{aligned}
\]
where $a$ and $b$ are defined in
Theorem~\ref{thm:main}.
Every exponent is greater than one.
Hence $\eta_k$ and $q_k$ are summable, as are
$\alpha_k^2/\gamma_k$ and $\gamma_k^2$.

Finally, all states at each finite iteration have finite
second moments.
Indeed, the initial error $e_0=y_0-c(x_0)$ has a finite
second moment because $c$ is globally Lipschitz.
The direction estimates, conditional observation bounds,
and exact tracking recursion then propagate finite
second moments by induction.
Thus any finite prefix preceding $k_0$ is integrable.
\end{proof}

\subsection{From boundedness to convergence of the iterates}

We use the almost supermartingale theorem of
Robbins and Siegmund \citep{robbins1971}.
In the form needed here, let $X_k$ be a nonnegative
integrable process adapted to $\Fcal_k$ and suppose
\[
\E_k X_{k+1}
\le(1+a_k)X_k-b_k+c_k,
\]
where $a_k$, $b_k$, and $c_k$ are nonnegative
$\Fcal_k$-measurable scalar coefficients.
If $\sum_k a_k<\infty$ and $\sum_k c_k<\infty$
almost surely, then $X_k$ has a finite almost sure limit
and $\sum_k b_k<\infty$ almost surely.
The form with random summable coefficients follows
by localization.

The next lemma gives the passage from a perturbed
iteration to the ideal dynamics.
It makes explicit the limiting trajectory argument
used in stochastic approximation
\citep{benaim1999,kushner2003}.

\begin{lemma}[Limiting trajectories]
\label{lem:limit}
Suppose Assumption~\ref{ass:convex} holds.
On a fixed sample path, let
$z_k\in\R^n\times\Ccal$ be a bounded sequence satisfying
\begin{equation}
z_{k+1}
=z_k+\alpha_kF(z_k)
+r_k+\Delta M_k+\alpha_k b_k.
\label{eq:euler}
\end{equation}
Here $\alpha_k>0$ satisfies
$\alpha_k\to0$ and $\sum_k\alpha_k=\infty$.
The vectors $r_k,\Delta M_k,b_k\in\R^{n+m}$ satisfy
\[
\sum_k\|r_k\|<\infty,
~~
\sum_k\Delta M_k\text{ converges},
~~
b_k\to0.
\]
Thus $r_k$ is an absolutely summable perturbation,
$\Delta M_k$ is an increment whose vector series
converges, and $b_k$ is a vanishing perturbation of
the vector field.

If $V_*(z_k)$ has a limit for one KKT reference point,
every accumulation point of $z_k$ belongs to $\Zstar$.
If the distance energies have limits for every reference
point in a countable dense subset of $\Zstar$,
then $z_k$ converges to a single KKT point.
\end{lemma}

\begin{proof}
Define the cumulative optimization times by
\[
t_0\coloneqq0,
~~
t_k\coloneqq\sum_{j=0}^{k-1}\alpha_j
~(k\ge1).
\]
Let $\widetilde z$ be the linear interpolation satisfying
\[
\widetilde z(t)
\coloneqq
z_k+\frac{t-t_k}{\alpha_k}(z_{k+1}-z_k),
~~
t\in[t_k,t_{k+1}].
\]
Since $\Ccal$ is convex, this interpolation remains in
$\R^n\times\Ccal$.

Fix a time window length $H>0$ and consider windows
starting at $t_N$, with $N\to\infty$.
The tails of the absolutely convergent series
$\sum_k r_k$ tend to zero uniformly over subintervals.
The same property holds for partial sums of
$\sum_k\Delta M_k$ because that vector series converges.
The cumulative contribution of $\alpha_k b_k$ over
such a window is bounded by
\[
\left(H+\sup_{j\ge N}\alpha_j\right)
\sup_{j\ge N}\|b_j\|\longrightarrow0.
\]
Partial interpolation intervals satisfy the same bounds.

Boundedness of $z_k$ and continuity of $F$ make
$F(z_k)$ bounded.
Equation~\eqref{eq:euler} therefore gives
$\|z_{k+1}-z_k\|\to0$.
Consequently, replacing the piecewise constant states
$z_k$ by $\widetilde z(t)$ in the integral of $F$
introduces an error tending uniformly to zero on each
window of length $H$.

Let $\varphi_t(z)$ denote the ideal solution at time $t$
with initial state $z$, whose global existence is supplied
by Theorem~\ref{thm:ideal}.
The energy bound in that theorem places the ideal
solutions starting from the bounded sequence $z_N$
in a common compact set.
On a compact neighborhood containing these solutions
and the interpolation, $F$ is Lipschitz.
Gronwall's inequality then gives
\[
\sup_{0\le t\le H}
\|\widetilde z(t_N+t)-\varphi_t(z_N)\|
\longrightarrow0.
\]

Now let $\bar z$ be an accumulation point and choose
indices $N_j\to\infty$ such that $z_{N_j}\to\bar z$.
Continuous dependence of the ideal solutions on their
initial states implies
\[
\widetilde z(t_{N_j}+t)\longrightarrow\varphi_t(\bar z)
\]
uniformly on every fixed finite time interval.
If $V_*(z_k)$ converges, the vanishing interpolation
increments imply that the limiting trajectory
$\varphi_t(\bar z)$ has constant energy.
By \eqref{eq:energy_identity}, it remains in the zero
dissipation set.
The invariant set argument in
Theorem~\ref{thm:ideal} therefore gives
$\bar z\in\Zstar$.

For the second assertion, let $\mathcal Q$ be the
countable dense subset of $\Zstar$ appearing in the
hypothesis.
For any reference point $q=(x_q,u_q)\in\Zstar$, write
\[
V_q(x,u)\coloneqq
\frac12\|x-x_q\|^2
+\frac1{2\kappa}\|u-u_q\|_{D^{-1}}^2.
\]
Choose $q_j\in\mathcal Q$ with $q_j\to q$.
Since $z_k$ is bounded, the quadratic formula gives
\[
\sup_k|V_{q_j}(z_k)-V_q(z_k)|\longrightarrow0.
\]
Each sequence $V_{q_j}(z_k)$ has a limit, so this
uniform approximation implies that $V_q(z_k)$ also
has a limit.

Choose a KKT accumulation point $\bar z$ as the
reference point $q$.
The energy $V_{\bar z}(z_k)$ tends to zero along an
accumulation subsequence.
Its full limit is therefore zero, and $z_k\to\bar z$.
\end{proof}

\begin{proof}[Proof of Theorem~\ref{thm:main}]
Apply the almost supermartingale theorem to
\eqref{eq:joint_drift}.
It gives a finite almost sure limit of $W_k$ and
\begin{equation}
\sum_k\alpha_k\Dcal_*(z_k)<\infty,
~~
\sum_k\frac{\alpha_k}{\delta_k}\|e_k\|^2<\infty
~\text{almost surely}.
\label{eq:sums}
\end{equation}
Since $V_k\le W_k$, the primal and dual states are
almost surely bounded.
Taking expectations in \eqref{eq:joint_drift},
dropping its negative terms, and using
$\sum_k\eta_k<\infty$ and $\sum_kq_k<\infty$ also gives
\[
\sup_k\E W_k<\infty,
~~
\sup_k\E V_k<\infty.
\]

To control the unweighted tracking error, return to
\eqref{eq:e_drift}.
The sequence $\alpha_k^2/\gamma_k$ is summable.
Moreover,
\[
\sum_k\left[
C\frac{\alpha_k^2}{\gamma_k}(1+V_k)
+C\gamma_k^2
\right]<\infty
~\text{almost surely},
\]
because $V_k$ is bounded on each such sample path.
A second application of the almost supermartingale
theorem shows that $\|e_k\|^2$ has a finite limit and
\[
\sum_k\gamma_k\|e_k\|^2<\infty
~\text{almost surely}.
\]
Since $\sum_k\gamma_k=\infty$, the limit must be zero.
It follows that $y_k=c(x_k)+e_k$ is almost surely
bounded.
Also,
\[
V_k=W_k-w_k\|e_k\|^2
\]
has a finite almost sure limit.

For convergence of the tracking error in mean square,
take expectations in \eqref{eq:e_drift}.
Because $\alpha_k/\gamma_k\to0$ and
$\sup_k\E V_k<\infty$, for all sufficiently large $k$,
\[
\E\|e_{k+1}\|^2
\le
(1-\gamma_k/2)\E\|e_k\|^2
+C\frac{\alpha_k^2}{\gamma_k}
+C\gamma_k^2.
\]
The additive term divided by $\gamma_k$ tends to zero:
\[
C\frac{\alpha_k^2}{\gamma_k^2}+C\gamma_k\to0.
\]
For any tolerance $\epsilon>0$, it is therefore
eventually at most $\epsilon\gamma_k/2$.
Subtracting $\epsilon$ from the preceding recursion
and iterating the contraction gives
\[
\limsup_{k\to\infty}\E\|e_k\|^2\le\epsilon,
\]
because $\sum_k\gamma_k=\infty$.
As $\epsilon$ is arbitrary,
$\E\|e_k\|^2\to0$.
This also gives a uniform bound on these second moments.
Together with $\sup_k\E V_k<\infty$, it proves
\eqref{eq:moments}.

Next, Cauchy--Schwarz and \eqref{eq:sums} yield
\[
\sum_k\alpha_k\|e_k\|
\le
\left(\sum_k\alpha_k\delta_k\right)^{1/2}
\left(
\sum_k\frac{\alpha_k}{\delta_k}\|e_k\|^2
\right)^{1/2}
<\infty
~\text{almost surely}.
\]
The signal perturbations
\[
r_k\coloneqq
\alpha_k
\begin{pmatrix}
-J_c(x_k)^\top p_k\\
\kappa p_k
\end{pmatrix}
\]
are therefore absolutely summable, by the bound on
$J_c$ and \eqref{eq:noise}.

For the derivative noise, \eqref{eq:noise},
\eqref{eq:moments}, and $\sum_k\alpha_k^2<\infty$
give
\[
\sum_k\E\|\alpha_k\xi_k\|^2<\infty.
\]
Since $\E_k\xi_k=0$, the partial sums of
$\sum_k\alpha_k\xi_k$ form a martingale bounded in
mean square.
The martingale convergence theorem therefore gives
almost sure and mean square convergence of this series
\citep{kushner2003}.
Thus the iteration has the form \eqref{eq:euler}, with
the vectors $r_k$ defined above,
\[
\Delta M_k\coloneqq
\begin{pmatrix}
-\alpha_k\xi_k\\
0
\end{pmatrix},
~~
b_k\coloneqq0.
\]
Lemma~\ref{lem:limit} shows that every accumulation
point is a KKT point.

For each other fixed KKT reference point, repeat the
joint energy construction, allowing the proof constants
to depend on that point.
Its distance energy also has a finite almost sure limit.
The set $\Zstar$, as a subset of a finite dimensional
Euclidean space, has a countable dense subset.
Taking the intersection of the corresponding probability
one events allows the second part of
Lemma~\ref{lem:limit} to be applied on a single event.
Hence $(x_k,u_k)$ converges almost surely to one random
point $(x^\infty,u^\infty)\in\Zstar$.

Finally, $e_k\to0$ and continuity of the signal maps give
\[
\lambda_k
=P_D(u_k+D(c(x_k)+e_k))
\longrightarrow
P_D(u^\infty+Dc(x^\infty))
=u^\infty.
\]
Continuity and Proposition~\ref{prop:zeros} also give
$\mathcal R(x_k,u_k)\to0$ almost surely.
\end{proof}

\subsection{Accumulated dissipation}

The joint energy estimate also controls the accumulated
dissipation of the ideal vector field along the stochastic
iterates.
For a positive integer $T$, define the cumulative
optimization step
\[
A_T\coloneqq\sum_{k=0}^{T-1}\alpha_k.
\]

\begin{corollary}[Average dissipation]
\label{cor:rate}
Under the assumptions of Theorem~\ref{thm:main},
for each fixed KKT reference point there is a finite
constant $C_*$, independent of $T$, such that
\begin{equation}
\frac{1}{A_T}
\sum_{k=0}^{T-1}
\alpha_k\E\Dcal_*(z_k)
\le\frac{C_*}{A_T}
=O(T^{-1/4+\vartheta}).
\label{eq:rate}
\end{equation}
\end{corollary}

\begin{proof}
Let $k_0$ be the index in Lemma~\ref{lem:joint}.
Define the deterministic products
\[
\Pi_{k_0}\coloneqq1,
~~
\Pi_{k+1}\coloneqq
\prod_{j=k_0}^{k}(1+\eta_j)
~(k\ge k_0).
\]
They satisfy
\[
1\le\Pi_{k+1}\le
\exp\left(\sum_{j=k_0}^{\infty}\eta_j\right).
\]
Take expectations in \eqref{eq:joint_drift}, drop the
negative tracking error term, and divide by
$\Pi_{k+1}$.
Summation then gives, for $T>k_0$,
\[
\sum_{k=k_0}^{T-1}\alpha_k\E\Dcal_*(z_k)
\le
\exp\left(\sum_{k=k_0}^{\infty}\eta_k\right)
\left(
\E W_{k_0}
+\sum_{k=k_0}^{\infty}q_k
\right).
\]

The finite prefix is integrable as well.
Indeed, the energy identity gives
$\Dcal_*(z)=-\langle\nabla V_*(z),F(z)\rangle$,
where $\nabla V_*(z)$ is the gradient with respect
to the joint state.
The quadratic form of $V_*$ and
Lemma~\ref{lem:growth} imply
\[
0\le\Dcal_*(z)\le C V_*(z).
\]
The finite second moments already established therefore
allow the prefix $k<k_0$ to be absorbed into a constant
$C_*$.
This proves
$\sum_{k<T}\alpha_k\E\Dcal_*(z_k)\le C_*$.
Finally, the prescribed steps give
$A_T$ of order $T^{1-a}=T^{1/4-\vartheta}$,
which proves \eqref{eq:rate}.
\end{proof}

The bound in \eqref{eq:rate} concerns the energy
dissipation $\Dcal_*$ defined in \eqref{eq:dissipation}.
For a $\mu$-strongly convex objective,
\eqref{eq:strong} yields the additional estimate
\[
\frac1{A_T}\sum_{k=0}^{T-1}\alpha_k
\E\left[
\mu\|x_k-x_*\|^2
+\|d(x_k,u_k)\|_{D^{-1}}^2
\right]
\le\frac{C_*}{A_T}.
\]
The full equilibrium residual $\mathcal R$ converges
almost surely by Theorem~\ref{thm:main};
\eqref{eq:rate} does not provide a finite time rate
for that residual.
Appendix~\ref{app:raw} analyzes the complementary
approach of direct constraint sampling with increasing
batch sizes.

\section{Numerical studies}
\label{sec:numerics}
We examine three aspects of the analysis: the equilibrium bias caused by direct constraint sampling, convergence when the KKT set is nonunique, and the effect of recursive estimation with nonlinear constraint observations. We then compare REC with other stochastic primal--dual methods under a common observation budget. Throughout, we set $D\coloneqq I$ and $\kappa\coloneqq1$, where $I$ is the identity matrix of the appropriate dimension. Residuals are evaluated using the exact population functions through \eqref{eq:residual}. Reported means are computed from the measurements of individual runs.

\subsection{Bias and displacement of the equilibrium}
We first examine the scalar bias function in Proposition~\ref{prop:bias}. At a KKT point, a component of $s_*=u_*+\rho c(x_*)$ is positive for an active inequality with a positive multiplier, zero for a weakly active inequality, and negative for an inactive inequality. The bias as a function of $s$ therefore describes how these three cases respond to constraint noise.

Figure~\ref{fig:bias}a compares the formulas in Corollary~\ref{cor:distributions} with estimates based on $10^5$ independent antithetic pairs. We use symmetric two point noise taking values $\pm1$ and standard Gaussian noise, both with unit variance. For each sampled noise value $\varepsilon$, the antithetic estimate averages $[s+\varepsilon]_+$ and $[s-\varepsilon]_+$ before subtracting $[s]_+$. For the two point distribution, this average equals the theoretical bias exactly. For Gaussian noise, the maximum absolute discrepancy over the evaluated arguments is $2.13\times10^{-4}$. The results agree with the predicted distinction: bounded two point noise produces no bias when $|s|\ge1$, whereas Gaussian noise produces a positive bias at every finite $s$. At $s=0$, the respective biases are $1/2$ and $1/\sqrt{2\pi}$.

We next consider
\[
\min_{x\le0}\ \frac{x^2}{2}-x,
\]
whose unique KKT point is $(0,1)$. With constraint observation noise taking values $\pm\tau$ with equal probability, where $\tau\ge0$, the equilibrium of the directly sampled mean dynamics is
\begin{equation}
\bar x_\tau=\min\{0,(1-\tau)/2\},~~
\bar u_\tau=1-\bar x_\tau,~~
|\bar x_\tau\bar u_\tau|
=\max\{0,(\tau^2-1)/4\}.
\label{eq:phase_prediction}
\end{equation}
For $\tau\le1$, the noise preserves the positive branch at the KKT point, as required by Proposition~\ref{prop:bias}. For $\tau>1$, the calculation in Proposition~\ref{prop:counterexample} gives a displaced equilibrium that violates complementarity.

We compare direct sampling (RAW) with recursive constraint estimation (REC), which follows \eqref{eq:algorithm}. Each noise amplitude uses 32 paths and $20000$ updates. Figure~\ref{fig:bias}b compares the final measurements with \eqref{eq:phase_prediction}. At $\tau=2$, RAW has a mean final primal value of $-0.49665$ and a mean complementarity error of $0.74308$, close to the predicted values $-0.5$ and $0.75$. REC has a mean final primal value of $0.00542$ and a mean residual of $3.51\times10^{-4}$. Thus, in this example, the directly sampled iteration approaches the incorrect equilibrium predicted by the mean dynamics, whereas recursive estimation yields final states close to the original KKT point.

\begin{figure}[tbp]
\centering
\includegraphics[width=\linewidth]{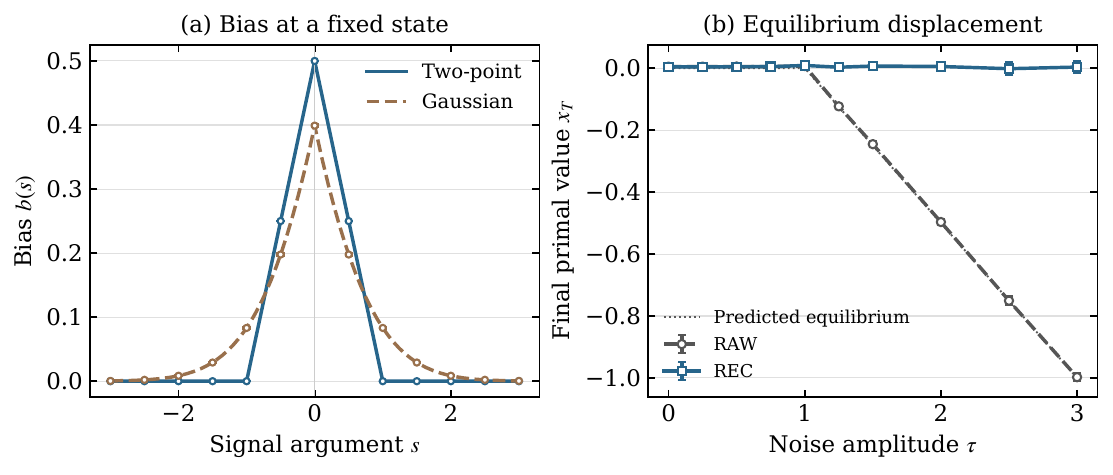}
\caption{Bias and equilibrium displacement. In panel (a), lines show the formulas in Corollary~\ref{cor:distributions}; markers show antithetic estimates, with error bars of $1.96$ estimated standard errors computed across independent pairs. In panel (b), points show final means over 32 paths, with error bars of one sample standard deviation. The dotted line shows the prediction in \eqref{eq:phase_prediction}.}
\label{fig:bias}
\end{figure}

\subsection{Convergence with nonunique solutions and multipliers}
To examine the setting of Theorem~\ref{thm:main} without uniqueness or strong convexity, let $Q\in\R^{24\times24}$ be an orthogonal matrix, write $v\coloneqq Qx$, and let $Q_8$ contain the first eight rows of $Q$. Define
\[
P\coloneqq Q_8^\top Q_8,~~
r\coloneqq Q^\top\tilde r,
\]
where $\tilde r\in\R^{24}$ has entries
\[
\tilde r_1=0.5,~~
\tilde r_2=1,~~
\tilde r_j=0.2~(3\le j\le8),~~
\tilde r_j=0~(9\le j\le24).
\]
Consider
\begin{equation}
f(x)\coloneqq\frac12\|P(x-r)\|^2,~~
c(x)\coloneqq(v_1,2v_1,v_2)^\top,~~
\K\coloneqq\{0\}^2\times\R_+.
\label{eq:nonunique_problem}
\end{equation}
The constraint $c(x)\in-\K$ imposes two redundant equalities and one inequality. The Hessian of the objective has rank eight, and the KKT set is
\begin{equation}
\Zstar=
\left\{
(x,u)\in\R^{24}\times\R^3:
v_1=v_2=0,\ 
v_3=\cdots=v_8=0.2,\ 
u_1+2u_2=0.5,\ 
u_3=1
\right\}.
\label{eq:nonunique_set}
\end{equation}
The coordinates $v_9,\ldots,v_{24}$ and the multiplier coordinate $(2u_1-u_2)/\sqrt5$ are unrestricted. We therefore measure the Euclidean distance to the full KKT set.

REC receives independent uniform noise in the gradient, Jacobian, and constraint observations. We use three initialization groups with different free primal and multiplier coordinates, eight paths per group, and an initial estimation error of norm five. The steps follow Theorem~\ref{thm:main}, with $\vartheta\coloneqq0.05$, $\alpha_0\coloneqq0.25$, $\gamma_0\coloneqq0.5$, and $\tau_0\coloneqq20$. Each path uses $30000$ updates. Appendix~\ref{app:experiments} specifies the initial states and observation bounds.

The mean distance to $\Zstar$ decreases from $1.5166$ initially to $0.0920$, $0.0463$, and $0.0231$ after $1000$, $10000$, and $30000$ updates, respectively. The final mean residual is $1.25\times10^{-3}$, and the mean squared estimation error is $1.13\times10^{-2}$. These measurements are consistent with the convergence of the primal--dual residual and tracking error in Theorem~\ref{thm:main}.

Distance to the KKT set does not measure movement along its free directions. Figure~\ref{fig:nonunique}b therefore also shows the free multiplier coordinate for the first replicate in each initialization group. Its final group means are $-1.265$, $0.763$, and $2.617$. To quantify movement in all free directions, we record the largest displacement from the start of each interval in the vector comprising the 16 free primal coordinates and the free multiplier coordinate. The mean displacements over $[1000,2000]$, $[5000,10000]$, and $[15000,30000]$ are $0.0464$, $0.0310$, and $0.0187$, respectively. Together with the decreasing distance to $\Zstar$, this reduced movement is consistent with convergence to individual KKT points, as asserted by Theorem~\ref{thm:main}, without requiring the runs to share a common limit.

\begin{figure}[tbp]
\centering
\includegraphics[width=\linewidth]{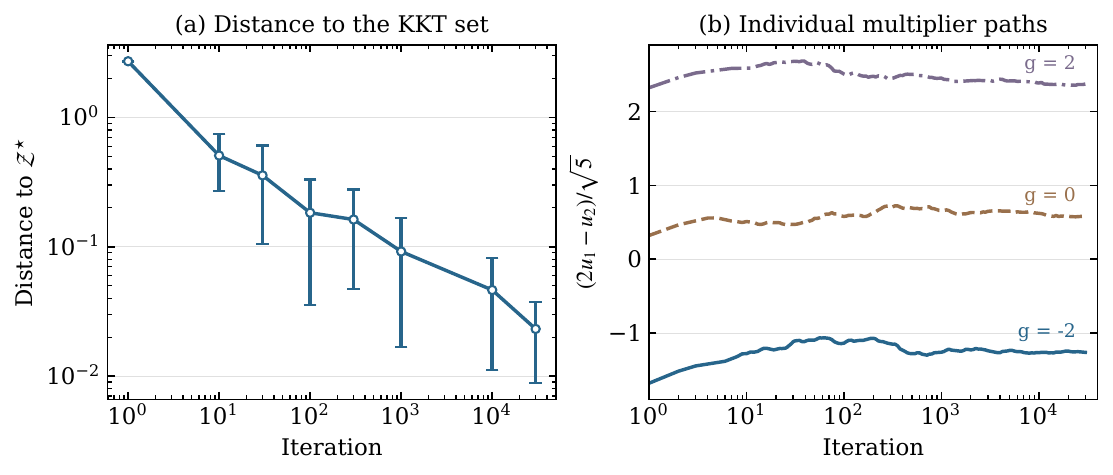}
\caption{REC on the problem in \eqref{eq:nonunique_problem}. Panel (a) shows means and one sample standard deviation over all 24 paths at the indicated checkpoints. Panel (b) shows the free multiplier coordinate $(2u_1-u_2)/\sqrt5$ for the first replicate in each initialization group, without smoothing or averaging. This coordinate is unrestricted in the KKT set \eqref{eq:nonunique_set}.}
\label{fig:nonunique}
\end{figure}

\subsection{Nonlinear constraints and estimation strategies}
We next consider four componentwise inequalities:
\begin{equation}
\begin{aligned}
f(x)&\coloneqq\frac12\|x-r\|^2,\\
c_j(x)&\coloneqq\psi(a_j^\top x)-\beta_j\le0,
~~ j=1,\ldots,4,\\
\psi(t)&\coloneqq\log(1+e^t),~~
\hat c_j(x)\coloneqq c_j(x)+\epsilon_j,
\end{aligned}
\label{eq:nonlinear_problem}
\end{equation}
where $x,r,a_j\in\R^n$, $\beta_j\in\R$, and the observation errors $\epsilon_j$ are independent and uniform on $[-\sigma,\sigma]$. Here $\K=\R_+^4$, and $\sigma$ is the noise amplitude. Gradient observations have independent component errors uniform on $[-0.1,0.1]$, and the constraint Jacobian is exact.

Let $A\in\R^{4\times n}$ have rows $a_j^\top$, each of unit norm. We generate three independent matrices of row rank four in each dimension $n\in\{20,50\}$. For each instance, a known KKT pair is constructed by setting
\[
\begin{gathered}
x_*\coloneqq0.2\mathbf1_n,~~
u_*\coloneqq(0.2,0.3,0.4,0.5)^\top,\\
\beta_j\coloneqq\psi(a_j^\top x_*),~~
r\coloneqq x_*+J_c(x_*)^\top u_*,
\end{gathered}
\]
where $\mathbf1_n$ is the vector of $n$ ones. The objective is strongly convex, and the constraint Jacobian is globally bounded and Lipschitz continuous, so these instances satisfy the regularity conditions used in Theorem~\ref{thm:main}. Since $\psi$ is strictly increasing, the feasible set can also be written as $Ax\le Ax_*$. The experiment therefore tests nonlinear constraint values and derivatives on a polyhedral feasible set.

We use noise amplitudes $\sigma\in\{0.5,2\}$ and eight paths for each instance and noise level, giving 96 runs per method. In addition to the current residual $\mathcal R(x,u)$, we measure the absolute objective error and feasibility error
\[
E_f(x)\coloneqq|f(x)-f(x_*)|,~~
V_c(x)\coloneqq\operatorname{dist}(c(x),-\K)
=\|[c(x)]_+\|.
\]
An observation tuple consists of one stochastic gradient vector and one stochastic constraint vector. Each run receives $20000$ tuples. RAW uses either one tuple per update or a batch of four, averaging both observation channels in the latter case. Thus, RAW with batch size four performs $5000$ updates, while the other configurations perform $20000$ updates. REC evaluates its constraint observation at the new primal point. A constant gain variant of REC, denoted CG, replaces the decreasing estimation gain by $\gamma_k\coloneqq0.1$.

To compare estimation strategies, RAW, CG, and REC use the same optimization step formula as a function of the update index. This comparison fixes the observation budget; the batched method consequently takes fewer optimization steps. CG and REC differ only in their estimation gains. The step parameters and complete results appear in Appendix~\ref{app:protocol} and Table~\ref{tab:estimation}.

Figure~\ref{fig:nonlinear}a reports the current residuals for $\sigma=2$. At the final budget, RAW with batch sizes one and four has mean residuals of $9.22\times10^{-2}$ and $1.44\times10^{-2}$, respectively. Batch averaging reduces the residual, while CG and REC give smaller values of $1.03\times10^{-3}$ and $3.87\times10^{-4}$. The corresponding mean squared tracking errors for CG and REC are $0.308$ and $0.0374$. These results are consistent with the role of tracking error in the convergence analysis: forming the augmented signal from a more accurate constraint estimate reduces the perturbation of the ideal dynamics.

At $\sigma=0.5$, the final mean residuals of CG and REC are much closer, at $6.29\times10^{-5}$ and $5.91\times10^{-5}$. Thus, the benefit of the decreasing estimation gain is more pronounced at the larger noise amplitude within the tested budget.

\subsection{Comparison with stochastic primal--dual methods}
We compare REC with the ordinary projected primal--dual method (PPD), the stochastic linearized proximal method of multipliers (SLPMM) of Zhang et al.\ \citep{zhang2022}, and the adaptive primal--dual stochastic gradient method (APriD) of Yan and Xu \citep{yan2022}. PPD places the decreasing optimization step inside the multiplier projection. SLPMM solves a strongly convex subproblem at each update, while APriD uses adaptive coordinate scaling. For SLPMM and APriD, we evaluate their prescribed averaged primal outputs as well as the residuals of their current primal--dual states.

Each method is calibrated over three step scales on two separate instances, using four paths per noise level and $5000$ observation tuples per run. The selection criterion is the mean of $E_f+V_c$, evaluated at the method's prescribed output. One scale per method is then fixed across all test instances and both noise levels. No test run is used for calibration. This differs from the estimation comparison above, which uses the REC optimization step formula for every configuration. Appendix~\ref{app:protocol} gives the parameter grid, selected scales, output definitions, and implementation checks. The common budget measures stochastic observations; it does not equate the computational cost of individual updates.

Figure~\ref{fig:nonlinear}b and Table~\ref{tab:comparison} report the results after calibration. REC has the smallest mean current residual at both noise levels. At $\sigma=0.5$, its residual is close to that of CG, while PPD has a smaller mean absolute objective error and feasibility error than REC. At $\sigma=2$, REC has a mean current residual of $3.87\times10^{-4}$, compared with $7.23\times10^{-4}$ for PPD and approximately $10^{-2}$ for SLPMM and APriD. This comparison concerns the current primal--dual state, which is the object of the point convergence result in Theorem~\ref{thm:main}.

The primal output measurements give a more mixed comparison. At $\sigma=2$, the prescribed averaged outputs of SLPMM and APriD have mean absolute objective errors of $0.00465$ and $0.00471$, respectively, compared with $0.00531$ for the current REC output. Their mean feasibility errors are $0.00883$ and $0.00841$, compared with $0.00838$ for REC. Thus, the smaller current residual of REC does not imply uniformly smaller primal output errors. Results for averaged primal outputs of every method are provided in Table~\ref{tab:averages}.

RAW also illustrates why feasibility must be considered together with the other measurements. At $\sigma=2$, its feasibility errors are zero at the reported precision, but its objective errors and current residuals remain substantially larger than those of REC. This pattern is consistent with the scalar example: feasibility alone does not identify a KKT point when direct sampling distorts the augmented multiplier signal.

\begin{figure}[tbp]
\centering
\includegraphics[width=\linewidth]{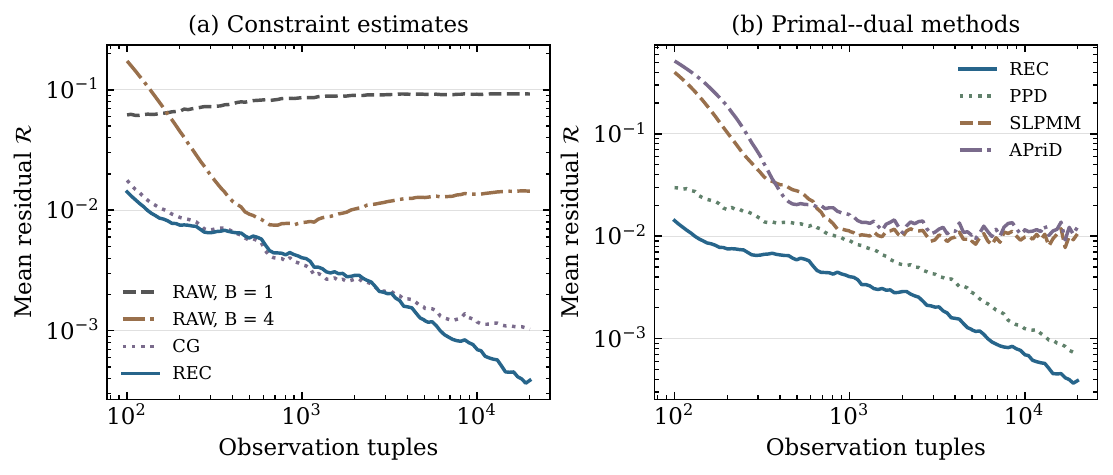}
\caption{Nonlinear study with noise amplitude $\sigma=2$. Curves show mean current residuals over 48 runs on six test instances. Panel (a) uses a common optimization step formula, indexed by update count, to compare constraint estimation strategies. Panel (b) uses the separately calibrated step parameters of each method. The horizontal budget counts observation tuples, each containing one stochastic gradient vector and one stochastic constraint vector; every sample in a batch is counted.}
\label{fig:nonlinear}
\end{figure}

\begin{table}[tbp]
\centering\small
\caption{Nonlinear comparison after $20000$ observation tuples using separately calibrated step parameters. Entries are means $\pm$ sample standard deviations over 48 runs per noise level. The residual $\mathcal R$ is evaluated at the current primal--dual state. The errors $E_f$ and $V_c$ are evaluated at the prescribed averaged primal output for SLPMM and APriD and at the current primal state for the other methods. Column headings indicate the numerical scaling.}
\label{tab:comparison}
\setlength{\tabcolsep}{4pt}
\begin{tabular}{clrrr}
\toprule
$\sigma$ & Method & $10^4\mathcal R$ & $10^2 E_f$ & $10^2 V_c$\\
\midrule
0.5 & RAW ($B=1$) & $14.27\pm6.34$ & $1.29\pm0.53$ & $0.29\pm0.39$\\
 & RAW ($B=4$) & $1.11\pm0.50$ & $0.22\pm0.15$ & $0.25\pm0.20$\\
 & CG & $0.63\pm0.34$ & $0.25\pm0.16$ & $0.55\pm0.21$\\
 & REC & $0.59\pm0.33$ & $0.23\pm0.15$ & $0.52\pm0.21$\\
 & PPD & $0.84\pm0.40$ & $0.17\pm0.13$ & $0.45\pm0.22$\\
 & SLPMM & $12.02\pm5.09$ & $0.35\pm0.14$ & $0.62\pm0.18$\\
 & APriD & $12.41\pm4.96$ & $0.34\pm0.14$ & $0.60\pm0.18$\\
\midrule
2 & RAW ($B=1$) & $932.77\pm215.93$ & $36.87\pm3.24$ & $0.00\pm0.00$\\
 & RAW ($B=4$) & $150.82\pm36.90$ & $9.59\pm1.14$ & $0.00\pm0.00$\\
 & CG & $10.28\pm5.08$ & $1.24\pm0.59$ & $0.29\pm0.38$\\
 & REC & $3.87\pm2.74$ & $0.53\pm0.40$ & $0.84\pm0.56$\\
 & PPD & $7.23\pm5.45$ & $0.55\pm0.41$ & $1.04\pm0.62$\\
 & SLPMM & $105.19\pm54.66$ & $0.46\pm0.35$ & $0.88\pm0.56$\\
 & APriD & $122.82\pm58.41$ & $0.47\pm0.35$ & $0.84\pm0.54$\\
\bottomrule
\end{tabular}
\end{table}

\section{Conclusion}
Constraint observations enter augmented primal--dual dynamics through a nonlinear multiplier signal. We characterized when this signal preserves the KKT equilibria and showed that direct sampling can produce a stable incorrect equilibrium. Recursive estimation yields a vanishing tracking error and, consequently, a vanishing perturbation of the augmented
multiplier signal. For smooth convex cone constraints, a joint energy argument establishes boundedness and almost sure convergence to a single KKT point. The numerical studies examine the predicted bias, convergence with nonunique solutions and multipliers, and the effect of constraint estimation in nonlinear problems. Further analysis could relax the global Jacobian bound and quantify convergence of the full residual over finite horizons.

\appendix
\renewcommand{\thetheorem}{\Alph{section}.\arabic{theorem}}
\renewcommand{\thesubsection}{\Alph{section}.\arabic{subsection}}

\section{Projection and augmented Lagrangian identities}
\label{app:projection}
Let $D\in\R^{m\times m}$ be symmetric positive definite, and let $\Ccal\subseteq\R^m$ be a closed convex cone. For $a,b\in\R^m$, write
\[
\langle a,b\rangle_{D^{-1}}\coloneqq a^\top D^{-1}b,
~~ \|a\|_{D^{-1}}\coloneqq\sqrt{a^\top D^{-1}a},
~~ \|a\|_D\coloneqq\sqrt{a^\top Da}.
\]
The metric projection $P_D(z)$ is the unique minimizer of $\|v-z\|_{D^{-1}}^2/2$ over $v\in\Ccal$. Its variational inequality is
\[
\langle z-P_D(z),v-P_D(z)\rangle_{D^{-1}}\le0
~~\text{for every }v\in\Ccal.
\]
Taking $v=0$ and $v=2P_D(z)$ gives
\[
\langle P_D(z),z-P_D(z)\rangle_{D^{-1}}=0.
\]
Consequently, completing the square yields
\[
\frac12\|P_D(z)\|_{D^{-1}}^2
=\max_{v\in\Ccal}
\left\{\langle v,z\rangle_{D^{-1}}-\frac12\|v\|_{D^{-1}}^2\right\}.
\]
The unique maximizer is $P_D(z)$. The value function is continuously differentiable, with Euclidean gradient $D^{-1}P_D(z)$; see the projection and squared distance identities in \citep{bauschke2017}. Thus, for
\[
\mathcal L_D(x,u)\coloneqq f(x)
+\frac12\|P_D(u+Dc(x))\|_{D^{-1}}^2
-\frac12\|u\|_{D^{-1}}^2,
\]
the chain rule gives
\[
\begin{aligned}
\nabla_x\mathcal L_D(x,u)
&=\nabla f(x)+J_c(x)^\top P_D(u+Dc(x))=G(x,u),\\
\nabla_u\mathcal L_D(x,u)
&=D^{-1}\bigl(P_D(u+Dc(x))-u\bigr)=D^{-1}d(x,u).
\end{aligned}
\]
These are the identities in \eqref{eq:augmented}.

For $z,w\in\R^m$, adding the two projection variational inequalities gives
\[
\|P_D(z)-P_D(w)\|_{D^{-1}}^2
\le\langle P_D(z)-P_D(w),z-w\rangle_{D^{-1}}.
\]
The Cauchy--Schwarz inequality therefore implies nonexpansiveness:
\[
\|P_D(z)-P_D(w)\|_{D^{-1}}\le\|z-w\|_{D^{-1}}.
\]
In particular, for a constraint perturbation $e\in\R^m$,
\[
\|P_D(u+D(c(x)+e))-P_D(u+Dc(x))\|_{D^{-1}}
\le\|De\|_{D^{-1}}=\|e\|_D.
\]
Equivalence of positive definite norms gives the Euclidean bounds used in Lemma~\ref{lem:growth}.

\section{Direct sampling with increasing batch sizes}
\label{app:raw}
REC keeps the sample counts fixed. This appendix analyzes an alternative in which the augmented signal is formed directly from a constraint estimate and the batch size increases. The same constraint observation enters both updates. This shared error can be controlled through its conditional second moment.

\subsection{Observation model and energy estimate}
Let $\Fcal_k$ be the history before iteration $k$, including $(x_k,u_k)$, and write $\E_k[\cdot]\coloneqq\E[\cdot\mid\Fcal_k]$. The observations are
\[
\hat c_k\coloneqq c(x_k)+\epsilon_k,~~
\hat g_k\coloneqq\nabla f(x_k)+\zeta_k,
\]
where $\epsilon_k\in\R^m$ and $\zeta_k\in\R^n$ are observation errors, $B_k$ is a deterministic positive integer, and $\sigma_c,\sigma_g\ge0$ are constants satisfying
\begin{equation}
\E_k\epsilon_k=\E_k\zeta_k=0,~~
\E_k\|\epsilon_k\|^2\le\sigma_c^2/B_k,~~
\E_k\|\zeta_k\|^2\le\sigma_g^2/B_k.
\label{eq:raw_noise}
\end{equation}
The two errors may be correlated. With an exact Jacobian, consider
\begin{equation}
\begin{aligned}
\hat\lambda_k&\coloneqq P_D(u_k+D\hat c_k),&\hat d_k&\coloneqq\hat\lambda_k-u_k,\\
\hat G_k&\coloneqq\hat g_k+J_c(x_k)^\top\hat\lambda_k,\\
x_{k+1}&\coloneqq x_k-\alpha_k\hat G_k,&u_{k+1}&\coloneqq u_k+\kappa\alpha_k\hat d_k.
\end{aligned}
\label{eq:raw}
\end{equation}
The steps $\alpha_k$ are deterministic and satisfy $0<\kappa\alpha_k\le1$. Assume $u_0\in\Ccal$ almost surely and $\E(\|x_0\|^2+\|u_0\|^2)<\infty$. The multiplier update is a convex combination of $u_k$ and $\hat\lambda_k$, so $u_k\in\Ccal$ for every $k$. Define $p_k\coloneqq\hat\lambda_k-\lambda(x_k,u_k)$. Nonexpansiveness gives
\begin{equation}
\hat d_k-d(x_k,u_k)=p_k,~~
\|p_k\|_{D^{-1}}\le\|\epsilon_k\|_D.
\label{eq:raw_error}
\end{equation}
Fix $(x_*,u_*)\in\Zstar$, write $c_*\coloneqq c(x_*)$, and define
\[
V_{*,k}\coloneqq\frac12\|x_k-x_*\|^2
+\frac1{2\kappa}\|u_k-u_*\|_{D^{-1}}^2.
\]
Although $p_k$ need not be centered, conditional centering of $\epsilon_k$ yields, with $\lambda_{\max}(D)$ denoting the largest eigenvalue of $D$,
\begin{equation}
\begin{aligned}
\E_k\langle\hat\lambda_k-u_*,\epsilon_k\rangle
&=\E_k\langle p_k,\epsilon_k\rangle\\
&\le\E_k\|\epsilon_k\|_D^2
\le\lambda_{\max}(D)\sigma_c^2/B_k.
\end{aligned}
\label{eq:shared}
\end{equation}
Using the scalar remainders $B_f$ and $Q_*$ and the
cone-valued remainder $B_c$ defined in
Lemma~\ref{lem:energy}, set
\[
\widehat{\mathcal H}_k\coloneqq B_f(x_*,x_k)+Q_*(x_k)
+\langle\hat\lambda_k,B_c(x_*,x_k)\rangle
+\tfrac12\|\hat d_k\|_{D^{-1}}^2\ge0.
\]
\begin{theorem}[Finite time bound for direct sampling]
\label{thm:raw_finite}
Under Assumptions~\ref{ass:convex}, \ref{ass:growth}, and \eqref{eq:raw_noise}, for a positive integer $T$, let $A_T\coloneqq\sum_{k=0}^{T-1}\alpha_k$ and $S_T\coloneqq\sum_{k=0}^{T-1}\alpha_k^2$. There are finite constants $C,C_o$, independent of $k,T$, such that
\begin{equation}
\frac{\sum_{k<T}\alpha_k\E\widehat{\mathcal H}_k}{A_T}
\le\frac{e^{CS_T}}{A_T}\left[
\E V_{*,0}+C_{\epsilon}\sum_{k<T}\frac{\alpha_k}{B_k}
+C_o\sum_{k<T}\frac{\alpha_k^2}{B_k}\right],
\label{eq:raw_bound}
\end{equation}
where $C_{\epsilon}\coloneqq\lambda_{\max}(D)\sigma_c^2$.
\end{theorem}
\begin{proof}
The projection optimality condition gives $c(x_k)+\epsilon_k-D^{-1}\hat d_k\in N_{\Ccal}(\hat\lambda_k)$, while $c_*\in N_{\Ccal}(u_*)$. Apply normal cone monotonicity to $\hat\lambda_k$ and $u_*$:
\[
\langle u_k-u_*,D^{-1}\hat d_k\rangle
\le\langle\hat\lambda_k-u_*,c(x_k)-c_*+\epsilon_k\rangle
-\|\hat d_k\|_{D^{-1}}^2.
\]
Expand the two quadratic terms of $V_*$ under \eqref{eq:raw} and use the convexity remainder identity from Lemma~\ref{lem:energy}. With $B_f=B_f(x_*,x_k)$, $B_c=B_c(x_*,x_k)$, and $Q_*=Q_*(x_k)$, this gives
\[
\begin{aligned}
V_{*,k+1}\le{}&V_{*,k}
-\alpha_k\bigl[B_f+Q_*+\langle\hat\lambda_k,B_c\rangle\bigr]\\
&-\alpha_k(1-\kappa\alpha_k/2)\|\hat d_k\|_{D^{-1}}^2
-\alpha_k\langle x_k-x_*,\zeta_k\rangle\\
&+\alpha_k\langle\hat\lambda_k-u_*,\epsilon_k\rangle
+\tfrac12\alpha_k^2\|\hat G_k\|^2.
\end{aligned}
\]
Since
\[
\hat G_k=G(x_k,u_k)+\zeta_k+J_c(x_k)^\top p_k,
\]
Lemma~\ref{lem:growth}, boundedness of $J_c$, and \eqref{eq:raw_error} give
\[
\E_k\|\hat G_k\|^2\le C V_{*,k}
+C(\sigma_g^2+\sigma_c^2)/B_k.
\]
This estimate does not require independence between $\epsilon_k$ and $\zeta_k$. Take conditional expectations, apply \eqref{eq:shared}, and use $1-\kappa\alpha_k/2\ge1/2$:
\begin{equation}
\E_k V_{*,k+1}\le(1+C\alpha_k^2)V_{*,k}
-\alpha_k\E_k\widehat{\mathcal H}_k
+C_{\epsilon}\frac{\alpha_k}{B_k}+C_o\frac{\alpha_k^2}{B_k}.
\label{eq:raw_drift}
\end{equation}
Set $P_0\coloneqq1$ and $P_{k+1}\coloneqq\prod_{j=0}^k(1+C\alpha_j^2)$. Dividing \eqref{eq:raw_drift} by $P_{k+1}$, taking expectations, and summing gives
\[
\sum_{k=0}^{T-1}\frac{\alpha_k\E\widehat{\mathcal H}_k}{P_{k+1}}
\le\E V_{*,0}
+\sum_{k=0}^{T-1}\frac{C_\epsilon\alpha_k/B_k+C_o\alpha_k^2/B_k}{P_{k+1}}.
\]
Because $1\le P_{k+1}\le e^{CS_T}$ and $\widehat{\mathcal H}_k\ge0$, this implies \eqref{eq:raw_bound}.
\end{proof}
For a fixed horizon with $\alpha_k\coloneqq a_0/\sqrt T$, $0<a_0\le1/\kappa$, and a fixed positive integer $B_k\coloneqq B$, the bound has order $T^{-1/2}+B^{-1}+(B\sqrt T)^{-1}$. The $B^{-1}$ term arises from the second moment of the shared constraint error. The bound controls the sampled energy dissipation measure $\widehat{\mathcal H}_k$. The term $B^{-1}$ is an upper-bound contribution from observation noise, not a lower bound on the attainable error.

\subsection{Almost sure convergence}
\begin{theorem}[Convergence with increasing batches]
\label{thm:raw_convergence}
Under the conditions of Theorem~\ref{thm:raw_finite}, suppose also that
\[
\begin{gathered}
\alpha_k\to0,~~\sum_k\alpha_k=\infty,~~\sum_k\alpha_k^2<\infty,\\
B_k\to\infty,~~\sum_k\alpha_k/B_k<\infty.
\end{gathered}
\]
Then \eqref{eq:raw} is almost surely bounded and converges to one KKT point.
\end{theorem}
\begin{proof}
The additive terms in \eqref{eq:raw_drift} are summable because $\sum_k\alpha_k/B_k<\infty$ and $B_k\ge1$. The Robbins--Siegmund theorem \citep{robbins1971} therefore gives a finite almost sure limit of $V_{*,k}$ for each fixed KKT reference point. Its quadratic form implies almost sure boundedness of $(x_k,u_k)$.

Define $\bar p_k\coloneqq\E_kp_k$ and $\widetilde p_k\coloneqq p_k-\bar p_k$. By \eqref{eq:raw_error} and conditional Jensen's inequality,
\[
\|\bar p_k\|\le C B_k^{-1/2}\longrightarrow0,
~~
\E_k\|\widetilde p_k\|^2\le C/B_k.
\]
Writing $z_k\coloneqq(x_k,u_k)$, the iteration becomes
\[
z_{k+1}=z_k+\alpha_kF(z_k)+\Delta M_k+\alpha_k b_k,
\]
where
\[
\Delta M_k\coloneqq\alpha_k
\begin{pmatrix}
-\zeta_k-J_c(x_k)^\top\widetilde p_k\\
\kappa\widetilde p_k
\end{pmatrix},
~~
b_k\coloneqq
\begin{pmatrix}
-J_c(x_k)^\top\bar p_k\\
\kappa\bar p_k
\end{pmatrix}.
\]
The increments satisfy $\E_k\Delta M_k=0$. Boundedness of $J_c$ gives
\[
\E_k\|\Delta M_k\|^2\le C\alpha_k^2/B_k,
~~ \|b_k\|\le C B_k^{-1/2}\longrightarrow0.
\]
Thus the partial sums of $\sum_k\Delta M_k$ form an $L^2$ bounded martingale and converge almost surely \citep{kushner2003}. Correlation between the two observation errors is allowed in this estimate. Lemma~\ref{lem:limit}, with zero summable remainder, shows that every accumulation point belongs to $\Zstar$. Applying the energy convergence argument on a countable dense subset of $\Zstar$ and intersecting the resulting probability one events gives convergence to a single KKT point.
\end{proof}
For example, $\alpha_k\coloneqq\alpha_0/(k+1)$ and
\[
B_k\coloneqq\left\lceil B_0[\log(k+2)]^{1+\eta}\right\rceil,~~ B_0>0,~ \eta>0,
\]
satisfy these conditions when $\alpha_0\le1/\kappa$. Here $\lceil\cdot\rceil$ denotes the ceiling function. A vanishing mean error is sufficient in Lemma~\ref{lem:limit} because its cumulative effect vanishes on each fixed optimization time window.

\section{Experimental details}
\label{app:experiments}
\subsection{Updates and measurements}
Throughout the experiments, $D\coloneqq I$ and $\kappa\coloneqq1$. The matrix $I$ has the same dimension as the multiplier vector. At iteration $k$, $\hat g_k$, $\hat J_k$, and $\hat c_k$ denote observations of $\nabla f(x_k)$, $J_c(x_k)$, and $c(x_k)$, respectively. RAW uses
\[
\begin{aligned}
\hat\lambda_k&\coloneqq P_I(u_k+\hat c_k),\\
x_{k+1}&\coloneqq x_k-\alpha_k(\hat g_k+\hat J_k^\top\hat\lambda_k),\\
u_{k+1}&\coloneqq u_k+\alpha_k(\hat\lambda_k-u_k).
\end{aligned}
\]
PPD uses
\[
x_{k+1}\coloneqq x_k-\alpha_k(\hat g_k+\hat J_k^\top u_k),
~~
u_{k+1}\coloneqq P_I(u_k+\alpha_k\hat c_k).
\]
For REC, define $\lambda_k\coloneqq P_I(u_k+y_k)$ and apply
\[
\begin{aligned}
x_{k+1}&\coloneqq x_k-\alpha_k(\hat g_k+\hat J_k^\top\lambda_k),\\
u_{k+1}&\coloneqq u_k+\alpha_k(\lambda_k-u_k),\\
y_{k+1}&\coloneqq(1-\gamma_k)y_k+\gamma_k\hat c_{k+1},
\end{aligned}
\]
where $\hat c_{k+1}$ is a fresh constraint observation at $x_{k+1}$. CG uses the same update with $\gamma_k\coloneqq0.1$. For batch size $B$, RAW averages $B$ independent observations in each sampled channel before forming its direction.

Measurements use the exact population functions. For a known optimum $x_*$, the objective, feasibility, complementarity, and tracking errors are
\[
\begin{aligned}
E_f(x)&\coloneqq|f(x)-f(x_*)|,
&V_c(x)&\coloneqq\operatorname{dist}(c(x),-\K),\\
C_c(x,u)&\coloneqq|u^\top c(x)|,
&E_y(x,y)&\coloneqq\|y-c(x)\|^2.
\end{aligned}
\]
For every method, the current residual is
\[
\begin{aligned}
\lambda(x,u)&=P_I(u+c(x)),\\
\mathcal R(x,u)&=\|\nabla f(x)+J_c(x)^\top\lambda(x,u)\|^2
+\|\lambda(x,u)-u\|^2.
\end{aligned}
\]
Statistics are calculated from these measurements on each run. Tables report means and sample standard deviations across the stated runs. Figures use no temporal smoothing. The step schedules used below are
\[
\alpha_k=\alpha_0(1+k/\tau_0)^{-(3/4+\vartheta)},
~~
\gamma_k=\gamma_0(1+k/\tau_0)^{-(1/2+\vartheta)},
\]
with parameters specified for each experiment.

\subsection{Scalar bias and equilibrium displacement}
The fixed-state calculation estimates
\[
b(s)=\E[s+\varepsilon]_+-[s]_+
\]
for symmetric two point noise $\varepsilon=\pm1$ and standard Gaussian noise. With seed 48101, generate $N\coloneqq10^5$ independent noise values $\varepsilon_i$ and compute
\[
Z_i(s)\coloneqq\frac{[s+\varepsilon_i]_++[s-\varepsilon_i]_+}{2}-[s]_+,
~~
\hat b_N(s)\coloneqq\frac1N\sum_{i=1}^N Z_i(s).
\]
The same draws are reused across the evaluated arguments $s$. Error bars have half-width $1.96\hat s_Z(s)/\sqrt N$, where $\hat s_Z(s)$ is the sample standard deviation of the $N$ pair averages. For two point noise, every pair average equals the exact expectation. For Gaussian noise, the largest absolute discrepancy from the formula in Corollary~\ref{cor:distributions} is $2.13\times10^{-4}$.

The dynamical experiment uses $f(x)=x^2/2-x$ and $c(x)=x\le0$, with exact derivatives and constraint noise $\pm\tau$. The ten amplitudes are
\[
\tau\in\{0,0.25,0.5,0.75,1,1.25,1.5,2,2.5,3\}.
\]
Each amplitude uses 32 paths and $20000$ updates. Seed 20260921 supplies paired paths for RAW, REC, and PPD. All methods start from $x_0=-1$ and $u_0=0$; REC additionally uses $y_0=-1$. The steps have $\alpha_0=0.25$, $\gamma_0=0.5$, $\vartheta=0.05$, and $\tau_0=1$. Figure~\ref{fig:bias} compares RAW and REC with the mean equilibrium prediction in \eqref{eq:phase_prediction}.

\subsection{Nonunique KKT set and observation bounds}
Generate an orthogonal matrix $Q\in\R^{24\times24}$ by a QR decomposition of a Gaussian matrix using seed 48102. Write $v=Qx$, let $q_j^\top$ be row $j$ of $Q$, and let $Q_8$ contain its first eight rows. Set
\[
P=Q_8^\top Q_8,~~ r=Q^\top\tilde r,~~
A=\begin{pmatrix}q_1^\top\\2q_1^\top\\q_2^\top\end{pmatrix},
\]
where $\tilde r\in\R^{24}$ has first two entries $0.5$ and $1$, entries three through eight equal to $0.2$, and all remaining entries zero. The problem is
\[
f(x)=\frac12\|P(x-r)\|^2,~~ c(x)=Ax\in-\K,
~~ \K=\{0\}^2\times\R_+.
\]
Stationarity and complementarity yield
\[
v_1=v_2=0,~ v_3=\cdots=v_8=0.2,~
u_1+2u_2=0.5,~ u_3=1,
\]
which defines the KKT set in \eqref{eq:nonunique_set}. Orthogonality gives its exact Euclidean distance:
\[
\operatorname{dist}((x,u),\Zstar)^2
=v_1^2+v_2^2+\sum_{j=3}^{8}(v_j-0.2)^2
+\frac{(u_1+2u_2-0.5)^2}{5}+(u_3-1)^2.
\]
The remaining 16 primal coordinates and $(2u_1-u_2)/\sqrt5$ are free.

Let $v_*\in\R^{24}$ have entries three through eight equal to $0.2$ and all other entries zero. In group $g\in\{-2,0,2\}$, initialize
\[
\begin{aligned}
v_0&=v_*+0.5\sum_{j=1}^{8}\mathbf e_j+g\mathbf e_9,
&x_0&=Q^\top v_0,\\
u_0&=(2g/\sqrt5,-g/\sqrt5,0.5)^\top,
&y_0&=Ax_0+(5/\sqrt3)\mathbf1_3.
\end{aligned}
\]
Here $\mathbf e_j$ is coordinate vector $j$ in $\R^{24}$ and $\mathbf1_3$ is the vector of three ones. Every group has initial tracking error of norm five and distance $\sqrt{2.3}\approx1.5166$ to the KKT set.

All observation errors have independent uniform components. The gradient and Jacobian amplitudes are $0.03$ and $0.01$, and the three constraint amplitudes are $0.2$, $0.4$, and $2$. Hence
\[
\begin{aligned}
\E\|\hat g_k-\nabla f(x_k)\|^2&=24(0.03)^2/3=0.0072,\\
\E\|\hat J_k-A\|_F^2&=72(0.01)^2/3=0.0024,\\
\E\|\hat c_{k+1}-Ax_{k+1}\|^2&=(0.2^2+0.4^2+2^2)/3=1.4,
\end{aligned}
\]
where $\|\cdot\|_F$ is the Frobenius norm and bounds the operator norm used in Assumption~\ref{ass:oracle}. The gradient has Lipschitz constant one and the Jacobian is constant. Each fresh constraint sample is independent of the current derivative observations. REC uses eight paths per group, $30000$ updates per path, seed 48103, and step parameters $\alpha_0=0.25$, $\gamma_0=0.5$, $\vartheta=0.05$, and $\tau_0=20$.

To measure movement in the free directions, define
\[
w_k\coloneqq\left(v_{k,9},\ldots,v_{k,24},
\frac{2u_{k,1}-u_{k,2}}{\sqrt5}\right)\in\R^{17},
~~
M_{p,q}\coloneqq\max_{p\le k\le q}\|w_k-w_p\|.
\]
Here $v_{k,j}$ and $u_{k,j}$ denote the corresponding coordinates at iteration $k$, and $p,q$ are integer endpoints. The mean values of $M_{p,q}$ for $[p,q]=[1000,2000]$, $[5000,10000]$, and $[15000,30000]$ are $0.0464$, $0.0310$, and $0.0187$, respectively. Every iterate in each interval enters this measurement.

\subsection{Nonlinear instances and calibration}
\label{app:protocol}
For $n\in\{20,50\}$, generate $A\in\R^{4\times n}$ with Gaussian entries and normalize each row to unit Euclidean norm. Write its rows as $a_j^\top$. The population problem is
\[
f(x)=\frac12\|x-r\|^2,~~
c_j(x)=\psi(a_j^\top x)-\beta_j\le0,~~
\psi(t)=\log(1+e^t).
\]
Set $x_*=0.2\mathbf1_n$, $u_*=(0.2,0.3,0.4,0.5)^\top$, $\beta_j=\psi(a_j^\top x_*)$, and $r=x_*+J_c(x_*)^\top u_*$. Then $c(x_*)=0$ and $\nabla f(x_*)+J_c(x_*)^\top u_*=0$. All generated matrices have row rank four. Since $\psi$ is strictly increasing, the feasible set is $Ax\le Ax_*$. The point
\[
\tilde x\coloneqq x_*-0.2A^\top(AA^\top)^{-1}\mathbf1_4
\]
strictly satisfies each constraint. Strong convexity gives a unique primal solution, and full row rank of $J_c(x_*)$ gives a unique multiplier.

Calibration uses one instance in each dimension, with matrix seeds 48201 and 48202. Testing uses seeds 48301--48306, the first three in dimension 20 and the remaining three in dimension 50. Each test instance uses eight paths at each noise amplitude $\sigma\in\{0.5,2\}$, giving 48 runs per noise level and 96 runs per method. Gradient noise is componentwise uniform on $[-0.1,0.1]$, constraint noise is componentwise uniform on $[-\sigma,\sigma]$, and $J_c$ is exact. All channels and coordinates are independent. Their second moments are
\[
\E\|\hat g-\nabla f(x)\|^2=n(0.1)^2/3,
~~
\E\|\hat c-c(x)\|^2=4\sigma^2/3.
\]
Since $0<\psi'(t)<1$ and $0<\psi''(t)\le1/4$,
\[
\|J_c(x)\|\le\|A\|,
~~
\|J_c(x)-J_c(z)\|\le\frac12\|x-z\|.
\]
Thus the global regularity and observation conditions for REC hold. All initial primal, multiplier, and tracking states are zero. SLPMM and APriD use the primal box $\mathcal X\coloneqq[-10,10]^n$. Every reported iterate and every SLPMM subproblem minimizer lies strictly inside this box, so the box constraint is inactive in the reported runs.

One observation tuple contains one gradient vector and one constraint vector. The constraint vector is evaluated at the new primal point for REC and CG and at the current primal point for the other methods. Every test run receives $20000$ tuples. RAW with batch size four averages four tuples per update and performs $5000$ updates; the other configurations perform $20000$ updates. No additional stochastic observations are used to evaluate the exact Jacobian or the reported measurements.

Calibration tests a scale $s\in\{0.25,1,4\}$ using four paths per calibration instance and noise level, with $5000$ tuples per run. For RAW, CG, REC, and PPD, the optimization steps are
\[
\alpha_k=0.25s(1+k/20)^{-0.8}.
\]
REC uses $\gamma_k=0.5(1+k/20)^{-0.55}$, while CG uses $\gamma_k=0.1$. SLPMM and APriD use the scale parameters specified below. Selection minimizes the mean of $E_f+V_c$ at each method's output over the calibration runs. One scale is then fixed for all test instances and noise levels. The selected scales are
\[
\begin{array}{lc}
\text{Method}&\text{Selected scale }s\\
\hline
\text{RAW }(B=1)&4\\
\text{RAW }(B=4),\ \text{SLPMM},\ \text{APriD}&1\\
\text{CG},\ \text{REC},\ \text{PPD}&0.25
\end{array}
\]
The estimation comparison instead fixes $s=0.25$ for RAW with both batch sizes, CG, and REC. Its common step formula is indexed by update count; batching reduces the number of updates under the fixed observation budget. CG and REC differ only in the estimation gain.

Noise seeds follow
\[
49100+1000h+\ell+500\mathbb I_{\{\sigma=2\}},
\]
where $h$ is the instance index, $\ell$ is the replicate index starting at zero, and $\mathbb I_{\{\sigma=2\}}$ equals one when $\sigma=2$ and zero otherwise. Calibration indices are 1 and 2, and test indices are 101--106. The same seeds are used across methods to supply paired observation errors at their respective query points.

\subsection{SLPMM and APriD implementations}
For SLPMM \citep{zhang2022}, let $T$ be the number of updates, $s$ the calibrated scale, $\sigma_s\coloneqq s/\sqrt T$, and $a_s\coloneqq\sqrt T/s$. With $J_k\coloneqq J_c(x_k)$, solve
\[
\begin{aligned}
\Delta_k\coloneqq\argmin_{\Delta\in\R^n:\,x_k+\Delta\in\mathcal X}
\biggl\{\hat g_k^\top\Delta+\frac{a_s}{2}\|\Delta\|^2
+\frac1{2\sigma_s}\|[u_k+\sigma_s(\hat c_k+J_k\Delta)]_+\|^2\biggr\},\\
x_{k+1}\coloneqq x_k+\Delta_k,~~
u_{k+1}\coloneqq[u_k+\sigma_s(\hat c_k+J_k\Delta_k)]_+.
\end{aligned}
\]
The four-constraint subproblems are solved in their dual variables to tolerance $10^{-12}$. Enumeration of all 16 active sets checks 20 subproblems, with maximum coordinate discrepancy below $1.1\times10^{-17}$. The test runs average 3.53 inner iterations per update. These inner iterations use the same sampled data and require no additional oracle calls. The prescribed primal output is
\[
\bar x_T^{\mathrm{SLPMM}}\coloneqq\frac1T\sum_{k=0}^{T-1}x_k.
\]

For APriD \citep{yan2022}, use $\beta_1=0.9$, $\beta_2=0.99$, and clipping threshold $\theta=10$. Initialize the first moment vector $m_0$, second moment vector $v_0$, and running maximum $\bar v_0$ to zero in $\R^n$. Define
\[
\begin{aligned}
h_k&\coloneqq\hat g_k+J_k^\top u_k,
&\tilde h_k&\coloneqq h_k/\max\{1,\|h_k\|/\theta\},\\
m_{k+1}&\coloneqq\beta_1m_k+(1-\beta_1)h_k,\\
v_{k+1}&\coloneqq\beta_2v_k+(1-\beta_2)\tilde h_k^{\odot2},
&\bar v_{k+1}&\coloneqq\max\{\bar v_k,v_{k+1}\},\\
x_{k+1}&\coloneqq\Pi_{\mathcal X}\left(x_k-\frac{0.1s}{\sqrt T}\,
 m_{k+1}\oslash\sqrt{\bar v_{k+1}}\right),\\
u_{k+1}&\coloneqq[u_k+(s/\sqrt T)\hat c_k]_+.
\end{aligned}
\]
Here $J_k=J_c(x_k)$, $\Pi_{\mathcal X}$ is Euclidean projection onto the box, $\odot2$ denotes componentwise squaring, and $\oslash$, the square root, and the maximum act componentwise. Projection onto a box is the same coordinatewise clipping operation for any positive diagonal metric. The continuous gradient noise makes all entries of the first second moment estimate positive almost surely. Clipping affects only the second moment estimate; the first moment uses $h_k$. With constant optimization steps, the source's dual-step recursion is constant. The prescribed output is
\[
\bar x_T^{\mathrm{APriD}}\coloneqq
\frac{\sum_{k=0}^{T-1}(1-\beta_1^{T-k})x_k}
{\sum_{k=0}^{T-1}(1-\beta_1^{T-k})}.
\]
An independent replay of 20 updates agrees within $3.5\times10^{-18}$.

\subsection{Additional numerical results}
Table~\ref{tab:averages} reports objective and feasibility errors of averaged primal states using the separately calibrated step scales. APriD uses its prescribed weights; every other method uses $T^{-1}\sum_{k=0}^{T-1}x_k$. These averages are additional measurements for RAW, CG, REC, and PPD, whose main comparison uses current primal outputs. Table~\ref{tab:estimation} reports current residuals and tracking errors with the common optimization step formula. Each noise level includes all six test instances and eight paths per instance. The displayed standard deviations describe variation across these 48 runs.

\begin{table}[tbp]
\centering\small
\caption{Averaged primal outputs with separately calibrated step scales. Means $\pm$ sample standard deviations cover 48 runs per noise level. APriD uses its prescribed weights; all other averages are arithmetic.}
\label{tab:averages}
\setlength{\tabcolsep}{4pt}
\begin{tabular}{clrr}
\toprule
$\sigma$ & Method & $10^2 E_f$ & $10^2 V_c$\\
\midrule
0.5 & RAW ($B=1$) & $1.17\pm0.16$ & $0.07\pm0.11$\\
 & RAW ($B=4$) & $0.11\pm0.09$ & $0.32\pm0.16$\\
 & CG & $0.76\pm0.17$ & $1.25\pm0.24$\\
 & REC & $0.68\pm0.18$ & $1.14\pm0.24$\\
 & PPD & $0.49\pm0.20$ & $0.95\pm0.28$\\
 & SLPMM & $0.35\pm0.14$ & $0.62\pm0.18$\\
 & APriD & $0.34\pm0.14$ & $0.60\pm0.18$\\
\midrule
2 & RAW ($B=1$) & $36.99\pm3.17$ & $0.00\pm0.00$\\
 & RAW ($B=4$) & $9.57\pm0.93$ & $0.00\pm0.00$\\
 & CG & $1.04\pm0.72$ & $0.55\pm0.50$\\
 & REC & $0.73\pm0.58$ & $0.84\pm0.66$\\
 & PPD & $0.81\pm0.60$ & $1.34\pm0.80$\\
 & SLPMM & $0.46\pm0.35$ & $0.88\pm0.56$\\
 & APriD & $0.47\pm0.35$ & $0.84\pm0.54$\\
\bottomrule
\end{tabular}
\end{table}

\begin{table}[tbp]
\centering\small
\caption{Constraint estimation with a common optimization step formula indexed by update count. Each noise level has 48 runs and a budget of $20000$ tuples. Batch size four uses $5000$ updates; the other configurations use $20000$. A dash indicates that there is no recursive estimate.}
\label{tab:estimation}
\setlength{\tabcolsep}{4pt}
\begin{tabular}{clrr}
\toprule
$\sigma$ & Estimate & $10^4\mathcal R$ & $10^2\|y-c(x)\|^2$\\
\midrule
0.5 & RAW ($B=1$) & $9.82\pm2.81$ & ---\\
 & RAW ($B=4$) & $1.41\pm0.52$ & ---\\
 & CG & $0.63\pm0.34$ & $1.82\pm1.29$\\
 & REC & $0.59\pm0.33$ & $0.18\pm0.12$\\
\midrule
2 & RAW ($B=1$) & $922.41\pm231.21$ & ---\\
 & RAW ($B=4$) & $143.86\pm35.12$ & ---\\
 & CG & $10.28\pm5.08$ & $30.77\pm22.02$\\
 & REC & $3.87\pm2.74$ & $3.74\pm2.25$\\
\bottomrule
\end{tabular}
\end{table}

\section{A cone constraint example}
\label{app:cone}
This experiment applies REC to a second order cone (SOC), whose projection couples the constraint coordinates. Let $x,r\in\R^{10}$ and define
\[
\begin{gathered}
\K\coloneqq\{(t,v)\in\R\times\R^9:t\ge\|v\|\},~~
J\coloneqq\diag(-1,1,\ldots,1)\in\R^{10\times10},\\
f(x)\coloneqq\tfrac12\|x-r\|^2,~~ c(x)\coloneqq Jx,~~
r\coloneqq(0,2,0,\ldots,0)^\top.
\end{gathered}
\]
The constraint is $c(x)\in-\K$. Its unique KKT pair is $x_*=u_*=(1,1,0,\ldots,0)^\top$: both feasibility and complementarity hold, and $x_*-r+J^\top u_*=0$. The SOC is self-dual, and its projection is
\[
P_I(t,v)=\begin{cases}(t,v),&\|v\|\le t,\\0,&\|v\|\le-t,\\
\dfrac{t+\|v\|}{2}\left(1,\dfrac{v}{\|v\|}\right),&\text{otherwise}.
\end{cases}
\]
Use $D=I_{10}$ and $\kappa=1$. REC uses an exact Jacobian and independent symmetric two point noise of amplitudes $0.1$ and $0.5$ in each gradient and constraint component. Their squared norm expectations are $0.1$ and $2.5$. The steps have $\alpha_0\coloneqq0.5$, $\gamma_0\coloneqq0.5$, $\vartheta\coloneqq0.05$, and $\tau_0\coloneqq1$. For $s\in\{1,10,100\}$, initialize
\[
x_0\coloneqq(-s,0,\ldots,0)^\top,~~
u_0\coloneqq(s,0,\ldots,0)^\top,~~ y_0\coloneqq s\mathbf1.
\]
The initial estimation error has norm $3s$. Each scale uses 32 paired paths, seed 20260922, and $10^5$ updates. The vector $\mathbf1$ has ten entries equal to one. At all three scales, the mean final residual is approximately $5.07\times10^{-5}$ and the mean squared estimation error is $1.26\times10^{-3}$. Their sample standard deviations are $2.71\times10^{-5}$ and $5.31\times10^{-4}$. These finite trajectories illustrate the conic setting covered by Theorem~\ref{thm:main}, with three different initial tracking errors.

\setlength{\bibsep}{3pt}
\bibliographystyle{elsarticle-num}
\bibliography{references}
\end{document}